\documentclass[journal]{IEEEtran}

\usepackage{extarrows,paralist}
\usepackage{mhchem} % Package for chemical equation typesetting
\usepackage{siunitx} % Provides the \SI{}{} command for typesetting SI units
\usepackage{graphicx} % Required for the inclusion of images
\usepackage{amsmath,graphicx}
\usepackage{mathtools}

\usepackage{amsthm}
\newtheorem{lem}{Lemma}
\newtheorem{theorem}{Theorem}

\newtheorem{assump}{Assumption}

\usepackage[square, comma, sort&compress, numbers]{natbib}
\usepackage{mathtools}
\usepackage{graphicx}
\usepackage{float}
\usepackage{caption}
\usepackage{color}
\usepackage{placeins}
\usepackage{float}
\usepackage{tabularx,colortbl}
\usepackage[skip=2pt,font=footnotesize]{caption}
\usepackage{amsmath,subfigure}
\usepackage{listings}
\usepackage[outdir=./]{epstopdf}
\usepackage{amssymb}
\usepackage{multirow}
\usepackage{algorithm}
\usepackage{algpseudocode}
\usepackage{hyperref}

\usepackage{enumitem}
\allowdisplaybreaks
\usepackage[dvipsnames]{xcolor}

\makeatletter
\newcommand{\vast}{\bBigg@{4}}
\newcommand{\Vast}{\bBigg@{5}}
\makeatother

\def\a{\alpha}
\def\b{\beta}

\def\mbb{\mathbb}%R
\def\mb{\mathbf}%vector
\def\mc{\mathcal}%set

\def\wt{\widetilde}

\def\bs{\boldsymbol}
\def\ol{\overline}
\def\bds{\boldsymbol}
\newcommand{\mn}[1]{{\left\vert\kern-0.25ex\left\vert\kern-0.25ex\left\vert\kern0.3ex #1 
		\kern0.3ex\right\vert\kern-0.25ex\right\vert\kern-0.25ex\right\vert}}
\def\bpi{\boldsymbol{\pi}}

\def\tsum{{\textstyle\sum}}
	
\begin{document}
\title{\bf\LARGE Distributed heavy-ball: A generalization and acceleration of first-order methods with gradient tracking}
\author{Ran Xin,~\emph{Student Member,~IEEE}, and Usman A. Khan,~\emph{Senior Member,~IEEE}
\thanks{
The authors are with the ECE Department at Tufts University, Medford, MA; {\texttt{ran.xin@tufts.edu, khan@ece.tufts.edu}}. This work has been partially supported by an NSF Career Award \# CCF-1350264.}
}
	
\maketitle
\begin{abstract}
We study distributed optimization to minimize a global objective that is a sum of smooth and strongly-convex local cost functions. Recently, several algorithms over undirected and directed graphs have been proposed that use a gradient tracking method to achieve linear convergence to the global minimizer. However, a connection between these different approaches has been unclear. In this paper, we first show that many of the existing first-order algorithms are in fact related with a simple state transformation, at the heart of which lies the~$\mc{AB}$ algorithm. We then describe \textit{distributed heavy-ball}, denoted as~$\mc{AB}m$, i.e.,~$\mc{AB}$ with momentum, that combines gradient tracking with a momentum term and uses nonidentical local step-sizes. By~simultaneously implementing both row- and column-stochastic weights,~$\mc{AB}m$ removes the conservatism in the related work due to doubly-stochastic weights or eigenvector estimation.~$\mc{AB}m$ thus naturally leads to optimization and average-consensus over both undirected and directed graphs, casting a unifying framework over several well-known consensus algorithms over arbitrary strongly-connected graphs. We show that~$\mathcal{AB}m$ has a global $R$-linear rate when the largest step-size is positive and sufficiently small. Following the standard practice in the heavy-ball literature, we numerically show that~$\mc{AB}m$ achieves accelerated convergence especially when the objective function is ill-conditioned.
\end{abstract}

\vspace{-0.1cm}

\vspace{-0.1cm}
\section{Introduction}\label{s1}
We consider distributed optimization, where~$n$ agents collaboratively solve the following problem:  
\begin{equation*}
\min_{\mb{x}\in\mathbb{R}^n}F(\mb{x}) \triangleq \frac{1}{n}\sum_{i=1}^{n}f_i(\mb{x}),
\end{equation*}
and each local objective,~$f_i:\mathbb{R}^p\rightarrow\mathbb{R}$, is smooth and strongly-convex. The goal of the agents is to find the global minimizer of the aggregate cost via only local communication with their neighbors. This formulation has recently received great interest with applications in e.g., machine learning~\cite{forero2010consensus,distributed_Boyd,raja2016cloud,wai2018multi}, control~\cite{jadbabaie2003coordination}, cognitive networks,~\cite{distributed_Mateos,distributed_Bazerque}, and source localization~\cite{distributed_Rabbit,safavi2018distributed}. 

Early work on this topic builds on the seminal work by Tsitsiklis in~\cite{DOPT1} and includes Distributed Gradient Descent (DGD)~\cite{uc_Nedic} and distributed dual averaging~\cite{duchi2012dual} over undirected graphs. Leveraging push-sum consensus~\cite{ac_directed0}, Refs.~\cite{opdirect_Tsianous,opdirect_Nedic} extend the DGD framework to directed graphs. Based on a similar concept, Refs.~\cite{D-DGD,D-DPS} propose Directed-Distributed Gradient Descent (D-DGD) for directed graphs that is based on surplus consensus~\cite{ac_Cai1}. In general, the DGD-based methods achieve sublinear convergence at~$\mathcal{O}\left(\frac{\log k}{\sqrt{k}}\right)$, where~$k$ is the number of iterations, because of the diminishing step-size used in the iterations. The convergence rate of DGD can be improved with the help of a constant step-size but at the expense of an inexact solution~\cite{DGD_Yuan,balancing}. Follow-up work also includes augmented Lagrangians~\cite{ADMM_Wei,ADMM_Mota,ADMM_Shi,ESOM}, which shows exact linear convergence for smooth and strongly-convex functions, albeit requiring higher computation at each iteration. 

To improve convergence and retain computational simplicity, fast first-order methods that do not (explicitly) use a dual update have been proposed. Reference~\cite{DNC} describes a distributed Nesterov-type method based on multiple consensus inner loops, at~$\mathcal{O}\left(\frac{\log k}{k^2}\right)$ for smooth and convex functions, with bounded gradients. EXTRA~\cite{EXTRA} uses the difference of two consecutive DGD iterates to achieve an~$\mathcal{O}\left(\frac{1}{k}\right)$ rate for arbitrary convex functions and a $Q$-linear rate for strongly-convex functions. DEXTRA~\cite{DEXTRA} combines push-sum~\cite{ac_directed0} and EXTRA~\cite{EXTRA} to achieve an $R$-linear rate over directed graphs given that a constant step-size is carefully chosen in some interval. Refs.~\cite{exactdiffusion1,exactdiffusion2} apply an adapt-then-combine structure~\cite{diffusion} to EXTRA~\cite{EXTRA} and generalize the symmetric weights in EXTRA to row-stochastic, over undirected graphs. 

Noting that DGD-type methods are faster with a constant step-size, recent work~\cite{AugDGM,harness,add-opt,diging,linear_row,FROST,AB,dnesterov,jakovetic2018unification,SUCAG} uses a constant step-size and replaces the local gradient, at each agent in DGD, with an estimate of the global gradient. A method based on gradient tracking was first shown in~\cite{AugDGM} over undirected graphs, which proposes Aug-DGM (that uses nonidentical step-sizes at the agents) with the help of dynamic consensus~\cite{DAC} and shows convergence for smooth convex functions. When the step-sizes are identical, the convergence rate of Aug-DGM was derived to be~$\mathcal{O}\left(\frac{1}{k}\right)$ for arbitrary convex functions and~$R$-linear for strongly-convex functions in~\cite{harness}. ADD-OPT~\cite{add-opt} extends~\cite{harness} to directed graphs by combining push-sum with gradient tracking and derives a contraction in an arbitrary norm to establish an~$R$-linear convergence rate when the global objective is smooth and strongly-convex. Ref.~\cite{diging} extends the analysis in~\cite{harness,add-opt} to time-varying graphs and establishes an~$R$-linear convergence using the small gain theorem~\cite{control}. In contrast to the aforementioned methods~\cite{AugDGM,harness,add-opt,diging}, where the weights are doubly-stochastic for undirected graphs and column-stochastic for directed graphs, FROST~\cite{linear_row,FROST} uses row-stochastic weights, which have certain advantages over column-stochastic weights. Ref.~\cite{jakovetic2018unification} unifies EXTRA~\cite{EXTRA} and gradient tracking methods~\cite{AugDGM,harness} in a primal-dual framework over static undirected graphs. More recently, Ref.~\cite{dnesterov} proposes distributed Nesterov over undirected graphs that also uses gradient tracking and shows a convergence rate of~$\mathcal{O}((1-{c}{\mathcal{Q}^{-\frac{5}{7}}})^k)$ for smooth, strongly-convex functions, where~$\mathcal{Q}$ is the condition number of the global objective.  Refs.~\cite{NEXT,sonata}, on the other hand, consider gradient tracking in distributed non-convex problems, while Ref.~\cite{SUCAG} uses second-order information to accelerate the convergence. 

Of significant relevance here is the~$\mathcal{AB}$ algorithm~\cite{AB},~also appeared later in~\cite{the_copy_work_2}, which can be viewed as a generalization of distributed first-order methods with gradient tracking. In particular, the algorithms over undirected graphs~in~Refs.~\cite{AugDGM,harness} are a special case of~$\mc{AB}$ because the doubly-stochastic weights therein are replaced by row- and column- stochastic weights. $\mathcal{AB}$ thus is naturally applicable to arbitrary directed graphs. Moreover, the use of both row- and column-stochastic weights removes the need for eigenvector estimation\footnote{Simultaneous application of both row- and column-stochastic weights was first employed  for average-consensus in~\cite{ac_Cai1} and towards distributed optimization in~\cite{D-DPS,D-DGD}, albeit without gradient tracking.}, required earlier in~\cite{add-opt,diging,linear_row,FROST}. Ref.~\cite{AB} derives an $R$-linear rate for~$\mathcal{AB}$ when the objective functions are smooth and strongly-convex. In this paper, we provide an improved understanding of~$\mc{AB}$ and extend it to the~$\mc{AB}m$ algorithm, a \textit{distributed heavy-ball method}, applicable to both undirected and directed graphs. We now summarize the main contributions:
\vspace{-0.05cm}
\begin{enumerate}
	\item We show that many of the existing accelerated first-order methods are either a special case of~$\mc{AB}$~\cite{AugDGM,harness}, or can be adapted from its equivalent forms~\cite{diging,add-opt,linear_row,FROST}.
	\item We propose a distributed heavy-ball method, termed as~$\mc{AB}m$, that combines~$\mc{AB}$ with a heavy-ball (type) momentum term. To the best of our knowledge, this paper is the first to use a momentum term based on the heavy-ball method in distributed optimization.
	\item $\mc{AB}m$ employs nonidentical step-sizes at the agents and thus its analysis naturally carries to nonidentical step-sizes in~$\mc{AB}$ and to the related algorithms in~\cite{AugDGM,harness,diging,add-opt,linear_row,FROST}.  
	\item We cast a unifying framework for consensus over arbitrary graphs that results from~$\mc{AB}m$ and subsumes several well-known algorithms~\cite{ac_Cai1,ac_row}. 
\end{enumerate}
\vspace{-0.05cm}
On the analysis front, we show that~$\mc{AB}$ (without momentum) converges faster as compared to the algorithms over directed graphs in~\cite{add-opt,diging,linear_row,FROST}, where separate iterations for eigenvector estimation are applied nonlinearly to the underlying algorithm. Towards~$\mc{AB}m$, we establish a \textit{global} $R$-linear convergence rate for smooth and strongly-convex objective functions when the largest step-size at the agents is positive and sufficiently small. This is in contrast to the earlier work on non-identical step-sizes within the framework of gradient tracking~\cite{AugDGM,digingun,digingstochastic,lu2018geometrical}, which requires the heterogeneity among the step-sizes to be sufficiently small, i.e., the step-sizes are close to each other. We also acknowledge that similar to the centralized heavy-ball method~\cite{polyak1964some,polyak1987introduction}, dating back to more than~50 years, and the recent work~\cite{ghadimi2015global,IAGM,lessard2016analysis,drori2014performance,polyak2017lyapunov,IGM,sHB}, a \textit{global} acceleration can only be shown via numerical simulations. Following the standard practice, we provide simulations to verify that~$\mathcal{AB}m$ has accelerated convergence, the effect of which is more pronounced when the global objective function is ill-conditioned.

We now describe the rest of the paper. Section~\ref{s2} provides preliminaries, problem formulation, and introduces distributed heavy-ball, i.e., the~$\mc{AB}m$ algorithm. Section~\ref{s3} establishes the connection between~$\mc{AB}$ and related algorithms. Section~\ref{s4} includes the main results on the convergence analysis, whereas Section~\ref{s6} provides a family of average-consensus algorithms that result naturally from~$\mc{AB}m$. Finally, Section~\ref{s7} provides numerical experiments and Section~\ref{s8} concludes the paper.

\textbf{Basic Notation:} We use lowercase bold letters to denote vectors and uppercase letters for matrices. The matrix,~$I_n$, is the~$n\times n$ identity, whereas~$\mb{1}_n$ ($\mb{0}_n$) is the~$n$-dimensional column vector of all ones (zeros). For an arbitrary vector,~$\mb{x}$, we denote its~$i$th element  by~$[\mb{x}]_i$ and its largest and smallest element by~$[\mb{x}]_{\max}$ and~$[\mb{x}]_{\min}$, respectively. We use~$\mbox{diag}(\mb{x})$ to denote a diagonal matrix that has~$\mb{x}$ on its main diagonal. For two matrices,~$X$ and~$Y$,~$\mbox{diag}\left(X,Y\right)$ is a block-diagonal matrix with~$X$ and~$Y$ on its main diagonal, and~$X\otimes Y$ denotes their Kronecker product. The spectral radius of a matrix,~$X$, is represented by~$\rho(X)$. For a primitive, row-stochastic matrix,~$A$, we denote its left and right eigenvectors corresponding to the eigenvalue of~$1$ by~$\bs{\pi}_r$ and~$\mb{1}_n$, respectively, such that~$\bs{\pi}_r^\top\mb{1}_n = 1$; similarly, for a primitive, column-stochastic matrix,~$B$, we denote its left and right eigenvectors corresponding to the eigenvalue of~$1$ by~$\mb{1}_n$ and~$\bs{\pi}_c$, respectively, such that~$\mb{1}_n^\top\bs{\pi}_c = 1$. For a matrix~$X$, we denote~$X_\infty$ as its infinite power (if it exists), i.e.,~$X_\infty =\lim_{k\rightarrow\infty}X^k.$ From the Perron-Frobenius theorem~\cite{matrix}, we have~$A_\infty=\mb{1}_n\bs{\pi}_r^\top$ and~$B_\infty=\bs{\pi}_c\mb{1}_n^\top$. We denote~$\left\|\cdot\right\|_\mathcal{A}$ and~$\left\|\cdot\right\|_\mathcal{B}$ as some arbitrary vector norms, the choice of which will be clear in Lemma~\ref{contra}, while~$\left\|\cdot\right\|$ denotes the Euclidean matrix and vector norms. 

\section{Preliminaries and Problem Formulation}\label{s2}
Consider~$n$ agents connected over a directed graph,~$\mc{G}=(\mc{V},\mc{E})$, where~$\mc{V}=\{1,\cdots,n\}$ is the set of agents, and~$\mc{E}$ is the collection of ordered pairs,~$(i,j),i,j\in\mc{V}$, such that agent~$j$ can send information to agent~$i$, i.e.,~$j\rightarrow i$. We define~$\mc{N}_i^{{\scriptsize \mbox{in}}}$ as the collection of in-neighbors of agent~$i$, i.e., the set of agents that can send information to agent~$i$. Similarly,~$\mc{N}_i^{{\scriptsize \mbox{out}}}$ is the set of out-neighbors of agent~$i$. Note that both~$\mc{N}_i^{{\scriptsize \mbox{in}}}$ and~$\mc{N}_i^{{\scriptsize \mbox{out}}}$ include agent~$i$. The agents solve the following  problem:
\begin{align}
\mbox{P1}:
\quad\min_{\mb{x}\in\mathbb{R}^n}F(\mb{x})\triangleq\frac{1}{n}\sum_{i=1}^nf_i(\mb{x}),\nonumber
\end{align}
where each~$f_i:\mbb{R}^p\rightarrow\mbb{R}$ is known only to agent~$i$. We formalize the set of assumptions as follows.
\begin{assump}\label{asp1}
	The  graph,~$\mc{G}$, is strongly-connected.
\end{assump}

\begin{assump}\label{asp2}
	Each local objective,~$f_i$, is~$\mu_i$-strongly-convex, i.e.,~$\forall i\in\mc{V}$ and~$\forall\mb{x}, \mb{y}\in\mbb{R}^p$, we have
	\begin{equation*}
	f_i(\mb{y})\geq f_i(\mb{x})+\nabla f_i(\mb{x})^\top(\mb{y}-\mb{x})+\frac{\mu_i}{2}\|\mb{x}-\mb{y}\|^2,
	\end{equation*}
	where~$\mu_i\geq0$ and~$\sum_{i=1}^{n}\mu_i>0$.
\end{assump} 

\begin{assump}\label{asp3}
Each local objective,~$f_i$, is~$l_i$-smooth, i.e., its gradient is Lipschitz-continuous:~$\forall i\in\mc{V}$ and~$\forall\mb{x}, \mb{y}\in\mbb{R}^p$, we have, for some~$l_i>0$,
	\begin{equation*}
	\qquad\|\mb{\nabla} f_i(\mb{x})-\mb{\nabla} f_i(\mb{y})\|\leq l_i\|\mb{x}-\mb{y}\|.
	\end{equation*}
\end{assump}
Assumptions~\ref{asp2} and~\ref{asp3} ensure that the global minimizer,~$\mb{x}^*\in\mbb{R}^p$, of~$F$ exists and is unique~\cite{nesterov2013introductory}. In the subsequent analysis, we use~$\mu \triangleq \frac{1}{n}\sum_{i=1}^{n}\mu_i$ and~$l \triangleq \frac{1}{n}\sum_{i=1}^{n} l_i$, as the strong-convexity and Lipschitz-continuity constants, respectively, for the global objective,~$F$. We define~$\ol{l}\triangleq\max_il_i$. We next describe the heavy-ball method that is credited to Polyak and then introduce the distributed heavy-ball method, termed as the~$\mc{AB}m$ algorithm, to solve Problem P1.

\subsection{Heavy-ball method}\label{hbm}
It is well known~\cite{polyak1987introduction,nesterov2013introductory} that the best achievable convergence rate of the gradient descent algorithm,
\begin{equation*}
\mb{x}_{k+1} = \mb{x}_{k} - \alpha\nabla F\left(\mb{x}_{k}\right), 
\end{equation*}
is~$\mc{O}((\tfrac{\mc{Q}-1}{\mc{Q}+1})^k)$, where~$\mc{Q}\triangleq\tfrac{l}{\mu}$ is the condition number of the objective function,~$F$. Clearly, gradient descent is quite slow when~$\mc{Q}$ is large, i.e., when the objective function is ill-conditioned. The seminal work by Polyak~\cite{polyak1964some,polyak1987introduction} proposes the following heavy-ball method:
\begin{equation}\label{HB}
\mb{x}_{k+1} = \mb{x}_{k} - \alpha\nabla F(\mb{x}_{k}) + \beta(\mb{x}_{k}-\mb{x}_{k-1}),
\end{equation}
where~$\beta\left(\mb{x}_{k}-\mb{x}_{k-1}\right)$ is interpreted as a ``momentum'' term, used to accelerate the convergence process. Polyak shows that with a specific choice of~$\alpha$ and~$\beta$, the heavy-ball method achieves a \emph{local} accelerated rate of~$\mc{O}((\tfrac{\mc{\sqrt{Q}}-1}{\mc{\sqrt{Q}}+1})^k)$. By local, it is meant that the acceleration can only be analytically shown when~$\|\mb{x}_0-\mb{x}^*\|$ is sufficiently small. Globally, i.e., for arbitrary initial conditions, only linear convergence is established, while an analytical characterization of the acceleration is still an open problem, see related work in~\cite{ghadimi2015global,IGM,IAGM,sHB,polyak2017lyapunov}. Numerical analysis and simulations are often employed to show global acceleration, i.e., it is possible to tune~$\alpha$ and~$\beta$ such that the heavy-ball method is faster than gradient descent~\cite{drori2014performance,lessard2016analysis}. 

\subsection{Distributed heavy-ball: The~$\mathcal{AB}m$ algorithm}
Recall, that our goal is to solve Problem P1 when the agents, possessing only local objectives, exchange information over a strongly-connected directed graph,~$\mc{G}$. Each agent,~$i\in\mc{V}$, maintains two variables:~$\mb{x}^{i}_k$,~$\mb{y}^{i}_k\in\mbb{R}^p$, where~$\mb{x}^i_k$ is the local estimate of the global minimizer and~$\mb{y}^i_k$ is an auxiliary variable. The~$\mc{AB}m$ algorithm, initialized with arbitrary~$\mb{x}^i_0$'s, $\mb{x}_{-1}^i=\mb 0_p$ and~$\mb{y}^i_0=\nabla f_i(\mb{x}^i_0),\forall i\in\mc{V}$, is given by\footnote{We note that several variants of this algorithm can be extracted by considering an adapt-then-combine update, e.g.,~$\sum_{j=1}^{n}b_{ij}(\mb{y}^j_k+\nabla  f_i(\mb{x}^i_{k+1}\big)-\nabla f_i\big(\mb{x}^i_k))$, see~\cite{AB}, instead of the combine-then-adapt update that we have used here in Eq.~\eqref{ABb}. The momentum term in Eq.~\eqref{ABa} can also be integrated similarly. We choose one of the applicable forms and note that extensions to other cases follow from this exposition and the subsequent analysis.}:
\begin{subequations}\label{AB}
	\begin{align}
	\mb{x}^i_{k+1}=&\sum_{j=1}^{n}a_{ij}\mb{x}^{j}_k-\alpha_i\mb{y}^i_k
	+ \beta_i \left( \mb{x}_k^i - \mb{x}_{k-1}^i \right), 
	\label{ABa}
	\\
	\mb{y}^i_{k+1}=&\sum_{j=1}^{n}b_{ij}\mb{y}^j_k+\nabla \label{ABb} f_i\big(\mb{x}^i_{k+1}\big)-\nabla f_i\big(\mb{x}^i_k\big),
	\end{align}
\end{subequations}
where~$\alpha_i\geq 0$ and~$\beta_i\geq 0$ are respectively the local step-size and the momentum parameter adopted by agent~$i$.
The weights,~$a_{ij}$'s and~$b_{ij}$'s, are associated with the graph topology and satisfy the following conditions:
\begin{align*} 
a_{ij}&=\left\{
\begin{array}{rl}
>0,&j\in\mc{N}_i^{{\scriptsize \mbox{in}}},\\
0,&\mbox{otherwise},
\end{array}
\right.
\quad
\sum_{j=1}^na_{ij}=1,\forall i\begin{color}{black},\end{color} 
\end{align*} 
\begin{align*} 
b_{ij}&=\left\{
\begin{array}{rl}
>0,&i\in\mc{N}_j^{{\scriptsize \mbox{out}}},\\
0,&\mbox{otherwise},
\end{array}
\right.
\quad
\sum_{i=1}^nb_{ij}=1,\forall j. 
\end{align*}
Note that the weight matrix,~$A=\{a_{ij}\}$, in Eq.~\eqref{ABa} is~RS (row-stochastic) and the weight matrix,~$B=\{b_{ij}\}$ in Eq.~\eqref{ABb} is CS (column-stochastic), both of which can be implemented over undirected and directed graphs alike. Intuitively, Eq.~\eqref{ABb} tracks the average of local gradients,~$\frac{1}{n}\sum_{i=1}^{n}\nabla f_i(\mb{x}^i_k)$, see~\cite{AugDGM,harness,add-opt,diging,linear_row,FROST,AB,dnesterov,jakovetic2018unification,DAC}, and therefore Eq.~\eqref{ABa} asymptotically approaches the centralized heavy-ball, Eq.~\eqref{HB}, as the descent direction~$\mb{y}^i_k$ becomes the gradient of the global objective.

\textbf{Vector form}: For the sake of analysis, we now write~$\mc{AB}m$ in vector form. We use the following notation:
\begin{eqnarray*}
\mb{x}_k\triangleq
\left[
\begin{array}{c}
\mb{x}_{k}^1\\
\vdots\\
\mb{x}_{k}^n
\end{array}
\right],~\mb{y}_k\triangleq
\left[
\begin{array}{c}
\mb{y}_{k}^1\\
\vdots\\
\mb{y}_{k}^n
\end{array}
\right],~
\nabla\mb{f}(\mb{x}_k)\triangleq
\left[
\begin{array}{c}
\nabla f_1(\mb{x}_{k}^1)\\
\vdots\\
\nabla f_n(\mb{x}_{k}^n)
\end{array}
\right],
\end{eqnarray*}
all in~$\mathbb{R}^{np}$. Let~$\bs\a$ and~$\bs\b$ define the vectors of the step-sizes and the momentum parameters, respectively. We now define augmented weight matrices,~$\mc{A},\mc{B}$, and augmented step-size and momentum matrices,~$D_{\bds{\alpha}},D_{\bds{\beta}}$: 
\begin{align*}
\mc{A} &\triangleq A \otimes I_p,\qquad  D_{\bds{\alpha}}\triangleq \mbox{diag}(\bs\a)\otimes I_p,\\
\mc{B} &\triangleq B \otimes I_p,\qquad D_{\bds{\beta}}\triangleq\mbox{diag}(\bs\b)\otimes I_p, 
\end{align*} 
all in~$\mathbb{R}^{np\times np}$. Using the notation above,~$\mc{AB}m$ can be compactly written as:
\begin{subequations}\label{ABmv}
\begin{align}
\mb{x}_{k+1} &= \mc{A}\mb{x}_k - D_{\bds{\alpha}}\mb{y}_k + D_{\bds{\beta}}
\left(\mb{x}_k-\mb{x}_{k-1}\right)
, \label{ABmva}\\
\mb{y}_{k+1} &= \mc{B}\mb{y}_k + \nabla \mb{f}(\mb{x}_{k+1})-\nabla \mb{f}(\mb{x}_k), \label{ABmvb}.
\end{align}
\end{subequations}
We note here that when~$\beta_i=0,\forall i$,~$\mc{AB}m$ reduces to~$\mc{AB}$~\cite{AB}, albeit with two distinguishing features: (i) the algorithm in~\cite{AB} uses an identical step-size,~$\alpha$, at each agent; and (ii) Eq.~\eqref{ABb} in~\cite{AB} is in an adapt-then-combine form.

\section{Connection with existing first-order methods}\label{s3}
In this section, we provide a generalization of several existing methods that employ gradient tracking~\cite{AugDGM,harness,add-opt,diging,linear_row,FROST} and show that~$\mc{AB}$ lies at the heart of these approaches. To proceed, we rewrite the~$\mc{AB}$ updates below (without momentum)~\cite{AB}. 
\begin{subequations}\label{ABv}
\begin{align}
\mb{x}_{k+1} &= \mc{A}\mb{x}_k - \alpha\mb{y}_k, \label{ABva}\\
\mb{y}_{k+1} &= \mc{B}\mb{y}_k + \nabla \mb{f}(\mb{x}_{k+1})-\nabla \mb{f}(\mb{x}_k) \label{ABvb}.
\end{align}
\end{subequations}
Since~$\mc{AB}$ uses both RS and CS weights simultaneously, it is natural to ask how are the optimization algorithms that require the weight matrices to be doubly-stochastic (DS)~\cite{EXTRA,AugDGM,harness,diging}, or only CS~\cite{add-opt,diging}, or only RS~\cite{linear_row,FROST}, are related to each other. We discuss this relationship next. 

\noindent \textbf{Optimization with DS weights}: Refs.~\cite{AugDGM,harness,diging} consider the following updates, termed as Aug-DGM in~\cite{AugDGM} and DIGing in~\cite{diging}:
\begin{subequations}\label{harness}
	\begin{align}
	\mb{x}_{k+1} &= \mc{W}\mb{x}_k -\alpha\mb{y}_k, \label{harness_a}\\
	\mb{y}_{k+1} &= \mc{W}\mb{y}_k + \nabla \mb{f}(\mb{x}_{k+1})-\nabla \mb{f}(\mb{x}_k), \label{harness_b}
	\end{align}
\end{subequations}
where~$\mc{W} = W \otimes I_p$, and~$W$ is a DS weight matrix. Clearly, to obtain DS weights, the underlying graph must be undirected (or balanced) and thus the algorithm in Eqs.~\eqref{harness} is not applicable to arbitrary directed graphs. That~$\mc{AB}$ generalizes Eqs.~\eqref{harness} is straightforward as the DS weights naturally satisfy the RS requirement in the top update and the CS requirement in the bottom update, while the reverse is not true. Similarly, we note that a related algorithm, EXTRA~\cite{EXTRA}, is given by
\begin{equation*}
\mb{x}_{k+1} = (I+\mc{W})\mb{x}_{k} - \wt{\mc{W}}\mb{x}_{k-1} - \alpha\left(\nabla\mb{f}(\mb{x}_k)-\nabla\mb{f}(\mb{x}_{k-1})\right),
\end{equation*}
where the two weight matrices,~$\mc{W}$ and~$\wt{\mc{W}}$, must be symmetric and satisfy some other stringent requirements, see~\cite{EXTRA} for details. Eliminating the~$\mb{y}_k$-update in~$\mc{AB}$, we note that~$\mc{AB}$ can be written in the EXTRA format as follows:
\begin{align}\nonumber
\mb{x}_{k+1} =& \left(I+(\mc{A+B}-I)\right)\mb{x}_k \\\label{ABmEXTRA}
&- \left(\mc{B A}\right)\mb{x}_{k-1} 
- \alpha \left(\nabla \mb{f}(\mb{x}_{k})-\nabla \mb{f}(\mb{x}_{k-1})\right).
\end{align}
It can be seen that the linear convergence of~$\mc{AB}$ does not follow from the analysis in~\cite{EXTRA} as~$\mc{A+B}-I$ and~$\mc{BA}$ are not necessarily symmetric. Analysis of the~$\mc{AB}$ algorithm, therefore, generalizes that of EXTRA to non-doubly-stochastic and non-symmetric weight matrices.

\noindent \textbf{Optimization with CS weights}: We now relate~$\mc{AB}$ to ADD-OPT/Push-DIGing that only require CS weights~\cite{add-opt,diging}. Since~$\mc{B}$ is already CS in~$\mc{AB}$, it suffices to seek a state transformation that transforms~$\mc{A}$ from RS to CS, while respecting the graph topology. To this aim, let us consider the following transformation on the~$\mb{x}_k$-update in~$\mc{AB}$: $\wt{\mb{x}}_k\triangleq\Pi_r\mb{x}_k,$ where~$\Pi_r \triangleq \mbox{diag}(n\bpi_r)\otimes I_p$ and~$\bpi_r$ is the left-eigenvector of the RS weight matrix,~$A$, corresponding to the eigenvalue~$1$. The resulting transformed~$\mc{AB}$ is given by
\begin{subequations}\label{AB_addopt}
	\begin{align}\label{AB_addopta}
	\wt{\mb{x}}_{k+1} &= \wt{\mc{B}}~\wt{\mb{x}}_k - \alpha\Pi_r\mb{y}_k,\\\label{AB_addoptb}
	\mb{x}_{k+1} &= \left(\mbox{diag}(n\bpi_r)\otimes I_p\right)^{-1}\wt{\mb{x}}_{k+1},\\\label{AB_addoptc}
	\mb{y}_{k+1} &= \mc{B}\mb{y}_k + \nabla \mb{f}(\mb{x}_{k+1})-\nabla \mb{f}(\mb{x}_k),
	\end{align}
\end{subequations}
where it is straightforward to show that~$\mc{\wt{B}}= \Pi_r\mc{A}\Pi_r^{-1}$ is now CS and~$\wt{\mc{B}}\left(\bpi_r\otimes I_p\right)=\bpi_r \otimes I_p$. 

In order to implement the above equations, two different CS matrices~($\wt{\mc{B}}$ and~$\mc{B}$) suffice, as long as they are primitive and respect the graph topology. The second update requires the right-eigenvector of the CS matrix used in the first update, i.e.,~$\wt{\mc{B}}$. Since this eigenvector is not known locally to any agent, ADD-OPT/Push-DIGing~\cite{add-opt,diging} propose learning this eigenvector with the following iterations:~$\mb{w}_{k+1}=\wt{\mc{B}}\mb{w}_{k},\mb{w}_0=\mb{1}_{np}$. The algorithms provided in~\cite{add-opt,diging} essentially implement Eqs.~\eqref{AB_addopt}, albeit with two differences: (i) the same CS weight matrix is used in all updates; and, (ii) the division in Eq.~\eqref{AB_addoptb} is replaced by the estimated component,~$\mb{w}_{k+1}^i$, of the left-eigenvector at each agent. This nonlinearity causes stability issues in ADD-OPT/Push-DIGing, whereas their convergence compared to~$\mc{AB}$ is slower because such an eigenvector estimation is not needed in the latter on the account of using the RS weights. Furthermore, the local step-sizes are now given by~$n\alpha[\bpi_r]_i$ that shows that ADD-OPT/Push-DIGing should work with nonidentical step-sizes. 
 
 \noindent \textbf{Optimization with RS weights}: The state transformation technique discussed above also leads to an algorithm from~$\mc{AB}$ that only requires RS weights. Since~$\mc{A}$ in~$\mc{AB}$ is RS, a transformation now is imposed on the~$\mb{y}_k$-update and is given by~$\wt{\mb{y}}_k\triangleq \Pi_c^{-1}\mb{y}_k$, where~$\Pi_c \triangleq \mbox{diag}(\bpi_c)\otimes I_p,$ and~$\bpi_c$ is the right-eigenvector of the CS weight matrix,~$B$, corresponding to the eigenvalue~$1$. Equivalently,~$\mc{AB}$ is given by 
 \begin{subequations}
 	\begin{align}
 	\mb{x}_{k+1} 
 	&= \mc{A}\mb{x}_k -\alpha\Pi_c\wt{\mb{y}}_k, \label{AB_frosta}\\
 	\wt{\mb{y}}_{k+1} &= \wt{\mc{A}}\wt{\mb{y}}_k + \Pi_c^{-1}\left(\nabla \mb{f}(\mb{x}_{k+1})-\nabla \mb{f}(\mb{x}_k)\right) \label{AB_frostb},
 	\end{align}
 \end{subequations}
 where~$\wt{\mc{A}}=\Pi_c^{-1}\mc{B}\Pi_c$ is now RS and~$\left(\bpi_c^\top\otimes I_p\right)\wt{\mc{A}}=\bpi_c^\top\otimes I_p$. Since the above form of~$\mc{AB}$ cannot be implemented because~$\bpi_c$ is not locally known, an eigenvector estimation is used in FROST~\cite{linear_row,FROST} and the division in Eq.~\eqref{AB_frostb} is replaced with the appropriate estimated component of~$\bpi_c$. The observations on different weight matrices in the two updates, nonidentical step-sizes, stability, and convergence made earlier for ADD-OPT/Push-DIGing are also applicable here. 
 
In conclusion, the~$\mc{AB}$ algorithm has various equivalent representations and several already-known protocols can in fact be derived from these representations. In a similar way,~$\mc{AB}m$ leads to protocols that add momentum to Aug-DGM, ADD-OPT/Push-DIGing, and FROST. We will revisit the relationship and equivalence cast here in Sections~\ref{s6} and~\ref{s7}. In Section~\ref{s6}, we will show that both~$\mc{AB}$ and~$\mc{AB}m$ naturally provide a non-trivial class of average-consensus algorithms, a special case of which are~\cite{ac_row} and surplus consensus~\cite{ac_Cai1}. In Section~\ref{s7}, we will compare these algorithms numerically. 
 
 \section{Convergence Analysis}\label{s4}
 We now start the convergence analysis of the proposed distributed heavy-ball method,~$\mc{AB}m$. In the following, we first provide some auxiliary results borrowed from the literature. 
 
\subsection{Auxiliary Results}
The following lemma establishes contractions with RS and CS matrices under arbitrary norms~\cite{AB}; note thacontraction in the Euclidean norm is not applicable unless the weight matrix is DS as in~\cite{harness,diging}. A similar result was first presented in~\cite{add-opt} for CS matrices, and later in~\cite{linear_row,FROST} for RS matrices. 
 \begin{lem}\label{contra}
 	Consider the augmented weight matrices~$\mc{A}$ and~$\mc{B}$. There exist vector norms, denoted as~$\left\|\cdot\right\|_\mc{A}$ and~$\left\|\cdot\right\|_\mc{B}$, such that~$\forall\mb{x}\in\mbb{R}^{np}$,
 	\begin{align}\label{A_ctr}
 	\left\|\mc{A}\mb{x}-\mc{A}_\infty\mb{x}\right\|_\mc{A}&\leq\sigma_\mc{A}\left\|\mb{x}-\mc{A}_\infty\mb{x}\right\|_\mc{A},\\\label{B_ctr}
 	\left\|\mc{B}\mb{x}-\mc{B}_\infty\mb{x}\right\|_\mc{B}&\leq\sigma_\mc{B}\left\|\mb{x}-\mc{B}_\infty\mb{x}\right\|_\mc{B},
 	\end{align}
 	where~$0<\sigma_\mc{A}<1$ and~$0<\sigma_\mc{B}<1$ are some constants.  
 \end{lem}
 
The next lemma from~\cite{AB} states that the sum of~$\mb{y}_k^i$'s preserves the sum of local gradients. This is a direct consequence of the dynamic consensus~\cite{DAC} employed with CS weights in the~$\mb{y}_k$-update of~$\mc{AB}m$.
 \begin{lem}\label{sum}
 	$(\mb{1}_n^\top \otimes I_p) \mb{y}_k = (\mb{1}_n^\top \otimes I_p) \nabla\mb{f}(\mb{x}_k),\forall k$.
 \end{lem}
 
 The next lemma is standard in the convex optimization theory~\cite{bertsekas1999nonlinear}. It states that the distance to the optimizer contracts at each step in the standard gradient descent method.
 \begin{lem}\label{centr_d}
 	Let~$F$ be~$\mu$-strongly-convex and~$l$-smooth. For~$0<\alpha<\frac{2}{l}$, we have ~$$\left\|\mb{x}-\alpha\nabla F(\mb{x})-\mb{x}^*\right\| \leq\sigma_F\left\|\mb{x}-\mb{x}^*\right\|,$$ where~$\sigma_F=\max\left(\left|1-\mu \alpha\right|,\left|1-l\alpha \right|\right)$.
 \end{lem}
 
\newpage
Finally, we provide a result from nonnegative matrix theory.
 \begin{lem}\label{rho}(Theorem 8.1.29 in~\cite{matrix})
 	Let $X\in\mathbb{R}^{n\times n}$ be a nonnegative matrix and~$\mb{x}\in\mathbb{R}^{n}$ be a positive vector. If~$X\mb{x}<\omega\mb{x}$ with~$\omega>0$, then~$\rho(X)<\omega$. 
 \end{lem}
 
 \subsection{Main results}\label{s5}
The convergence analysis of~$\mc{AB}m$ is based on deriving a contraction relationship between the following four quantities: 
\begin{inparaenum}[(i)]
\item $\|\mb{x}_{k+1}-\mc{A}_\infty\mb{x}_{k+1}\|_\mc{A}$, the consensus error in the network; 
\item $\|\mc{A}_\infty\mb{x}_{k+1}-\mb{1}_n \otimes \mb{x}^*\|$, the optimality gap; 
\item $\left\|\mb{x}_{k+1}-\mb{x}_k\right\|$, the state difference; and 
\item $\|\mb{y}_{k+1}-\mc{B}_\infty\mb{y}_{k+1}\|_\mc{B}$, the (biased) gradient estimation error.
\end{inparaenum}
We will establish an LTI-system inequality where the state vector is the collection of these four quantities and then develop the convergence properties of the corresponding system matrix. Before we proceed, note that since all vector norms on finite-dimensional vector spaces are equivalent~\cite{matrix}, there exist positive constants~$c_{\mathcal{A}\mathcal{B}},c_{\mathcal{B}\mathcal{A}},c_{2\mathcal{A}},c_{\mathcal{A}2},c_{2\mathcal{B}},c_{\mathcal{B}2}$ such that
\begin{align*}
\|\cdot\|_\mathcal{A} &\leq c_{\mathcal{A}\mathcal{B}}\|\cdot\|_\mathcal{B},~~
\|\cdot\|   \leq c_{2\mathcal{A}}\|\cdot\|_\mathcal{A},~~
\|\cdot\|_\mathcal{A} \leq c_{\mathcal{A}2}\|\cdot\|,
\\
\|\cdot\|_\mathcal{B} &\leq c_{\mathcal{B}\mathcal{A}}\|\cdot\|_\mathcal{A},
~~\|\cdot\|  \leq c_{2\mathcal{B}}\|\cdot\|_\mathcal{B},
~~\|\cdot\|_\mathcal{B}  \leq c_{\mathcal{B}2}\|\cdot\|.
\end{align*}    
We also define~$\ol{\alpha}\triangleq\left[\bs{\alpha}\right]_{\max}$ and~$\ol{\beta}\triangleq\left[\bs{\beta}\right]_{\max}$. In the following, we first provide an upper bound on the estimate,~$\mb{y}_k$, of the gradient of the global objective that will be useful in deriving the aforementioned LTI system.
\begin{lem}\label{y}
	The following inequality holds,~$\forall k$:
	\begin{align*}
	\|\mb{y}_k\| \leq&~ c_{2\mc{A}}\ol{l}\left\|\mc{B}_{\infty}\right\|\|\mb{x}_k-\mc{A}_\infty\mb{x}_k\|_\mc{A} 
	+ c_{2\mc{B}}\|\mb{y}_k-\mc{B}_\infty\mb{y}_k\|_\mc{B}\\
	&+ \ol{l}\left\|\mc{B}_{\infty}\right\|\|\mc{A}_\infty\mb{x}_k-\mb{1}_n \otimes \mb{x}^*\|.
	\end{align*}
\end{lem}
\begin{proof}
	Recall that~$\mc{B}_{\infty}= (\bs{\pi}_c\otimes I_p)(\mb{1}_n^\top\otimes I_p).$ We have		
	\begin{equation}\label{inte_3}
	\left\|\mb{y}_k\right\| \leq c_{2\mc{B}}\left\|\mb{y}_k-\mc{B}_\infty\mb{y}_k\right\|_\mc{B} + \left\|\mc{B}_\infty\mb{y}_k\right\|.
	\end{equation}
	We next bound~$\left\|\mc{B}_\infty\mb{y}_k\right\|$:
	\begin{align}\label{inte_4}
	\|\mc B_\infty\mb{y}_k\| =&~ \|(\bs{\pi}_c\otimes I_p)(\mb{1}_n^\top\otimes I_p)\nabla\mb{f}(\mb{x}_k)\|\nonumber,\\
	=&~ \|\bs{\pi}_c\|\left\|{\tsum}_{i=1}^{n}\nabla f_i(\mb{x}_k^i)-{\tsum}_{i=1}^{n}\nabla f_i(\mb{x}^*)\right\| \nonumber,\\
%	\leq&~ \|\bs{\pi}_c\|~\ol{l}~\tsum_{i=1}^{n}\|\mb{x}_k^i-\mb{x}^*\| \nonumber,\\
	\leq&~\|\bs{\pi}_c\|~\ol{l}~\sqrt{n}\|\mb{x}_k-\mb{1}_n\otimes \mb{x}^*\|, \nonumber\\
	\leq&~c_{2\mc{A}}~\ol{l}~\left\|\mc{B}_{\infty}\right\| \|\mb{x}_k-\mc A_\infty\mb{x}_k\|_\mc{A} \nonumber\\
	&+~\ol{l}~\left\|\mc{B}_{\infty}\right\|\|\mc A_\infty\mb{x}_k-\mb{1}_n \otimes \mb{x}^*\|,
	\end{align}
	where the first inequality uses Jensen's inequality and the last inequality uses the fact that~$\left\|\mc{B}_{\infty}\right\|=\sqrt{n}\|\bs{\pi}_c\|$. The lemma follows by plugging Eq.~\eqref{inte_4} into Eq.~\eqref{inte_3}.
\end{proof}

In the next Lemmas~\ref{xc}-\ref{yc}, we derive the relationships among the four quantities mentioned above. We start with a bound on~$\|\mb{x}_{k+1}-\mc{A}_\infty\mb{x}_{k+1}\|_\mc{A}$, the consensus error in the network. 
\begin{lem}\label{xc}
The following inequality holds,~$\forall k$:
\begin{align*}
\|\mb{x}&_{k+1}-\mc{A}_\infty\mb{x}_{k+1}\|_\mc{A} \\
\leq&\left(\sigma_{\mc{A}}+\ol{\alpha} c_{\mc{A}2}c_{2\mc{A}}~\ol{l}\left\|I_{np}-\mc{A}_\infty\right\|\left\|\mc{B}_{\infty}\right\|\right) \left\|\mb{x}_{k}-\mc{A}_\infty\mb{x}_{k}\right\|_\mc{A} \\ 
&+ \ol{\alpha} c_{\mc{A}2}~\ol{l}\left\|I_{np}-\mc{A}_\infty\right\| \left\|B_{\infty}\right\|\|\mc{A}_\infty\mb{x}_k-\mb{1}_n \otimes \mb{x}^*\| \\ 
&+ \ol\beta c_{\mc{A}2}\left\|I_{np}-\mc{A}_\infty\right\| \left\|\mb{x}_k-\mb{x}_{k-1}\right\|\nonumber\\
&+ \ol{\alpha}c_{\mc{A}2}c_{2\mc{B}}\left\|I_{np}-\mc{A}_\infty\right\|\left\|\mb{y}_k-\mc{B}_\infty\mb{y}_k\right\|_\mc{B}.
\end{align*}
\end{lem}

\newpage
\begin{proof}
First, note that~$\mc{A}_{\infty}\mc{A}=\mc{A}_{\infty}$. Following the~$\mb{x}_k$-update of~$\mc{AB}m$ in Eq.~\eqref{ABmva} and using the one-step contraction property of~$\mc{A}$ from Lemma~\ref{contra}, we have:
\begin{align*}
\big\|\mb{x}&_{k+1}-\mc{A}_\infty\mb{x}_{k+1}\big\|_\mc{A} \nonumber\\
=& \big\|\mc{A}\mb{x}_k -D_{\bds{\alpha}}\mb{y}_k + D_{\bds{\b}} (\mb{x}_k-\mb{x}_{k-1})
\\&-A_\infty\big(\mc{A}\mb{x}_k -D_{\bds{\alpha}}\mb{y}_k +D_{\bds{\b}} (\mb{x}_k-\mb{x}_{k-1})\big)\big\|_\mc{A},	\nonumber\\
\leq&~\sigma_{\mc{A}} \left\|\mb{x}_{k}-\mc{A}_\infty\mb{x}_{k}\right\|_\mc{A} + \ol{\alpha}~c_{\mc{A}2}\left\|I_{np}-\mc{A}_\infty\right\| \left\|\mb{y}_k\right\|
\\&+ \ol{\beta}~c_{\mc{A}2}\left\|I_{np}-\mc{A}_\infty\right\| \left\|\mb{x}_k-\mb{x}_{k-1}\right\|, 
\end{align*}
and the proof follows from Lemma~\ref{y}.
\end{proof}
Next, we derive a bound for~$\left\|\mc{A}_\infty\mb{x}_{k+1}-\mb{1}_n\otimes \mb{x}^*\right\|$, which can be interpreted as the optimality gap between the network accumulation state,~$\mc{A}_\infty\mb{x}_k$, and the global minimizer,~$\mb{1}_n\otimes\mb{x}^*$.
\begin{lem}\label{xo}
The following inequality holds, $\forall k$, when

\noindent $0<\bpi_r^\top\mbox{diag}(\bds{\alpha})\bpi_c<\frac{2}{nl}$: 
	\begin{align}\label{2}
	\|\mc{A}_\infty\mb{x}&_{k+1}-\mb{1}_n \otimes \mb{x}^*\| \nonumber\\ 
	\leq&~\ol{\alpha}\left(\bpi_r^\top\bpi_{c}\right)n\ol{l}c_{2\mc{A}}\left\|\mb{x}_{k}-\mc{A}_\infty\mb{x}_{k}\right\|_\mc{A}  \nonumber\\
	&+\lambda \left\|  \mc{A}_{\infty}\mb{x}_k -\mb{1}_n \otimes \mb{x}^* \right\| \nonumber\nonumber\\
	&+ \beta \|\mc{A}_{\infty}\| \|\mb{x}_k-\mb{x}_{k-1}\| \nonumber \\ 
	&+ \ol{\alpha} c_{2B}\|\mc{A}_\infty\|\left\|\mb{y}_k-\mc{B}_{\infty}\mb{y}_k\right\|_\mc{B},
	\end{align}
	where~{\small$\lambda=\max\left\{\left|1-\mu n \bpi_r^\top\mbox{diag}(\bds{\alpha})\bpi_c\right|,
		\left|1-ln\bpi_r^\top\mbox{diag}(\bds{\alpha})\bpi_c \right|\right\}.
		$}
\end{lem}
\begin{proof}
	Recall the~$\mb{x}_k$-update of~$\mc{AB}m$ in Eq.~\eqref{ABmva}, we have
	that
	\begin{align}\label{21}
	\|\mc{A}_\infty\mb{x}&_{k+1}-\mb{1}_n \otimes \mb{x}^*\| \nonumber\\
	=& \left\|\mc{A}_\infty\big(\mc{A}\mb{x}_k -D_{\bds{\alpha}}\mb{y}_k + D_{\bds\b} (\mb{x}_k-\mb{x}_{k-1})\big)-\mb{1}_n \otimes \mb{x}^*\right\|, \nonumber\\
	=&~\big\|\mc{A}_\infty\big(\mc{A}\mb{x}_k -D_{\bds{\alpha}}\mb{y}_k + (D_{\bds{\alpha}}-D_{\bds{\alpha}})\mc{B}_{\infty}\mb{y}_k \nonumber\\
	&~~+ D_{\bds\b} (\mb{x}_k-\mb{x}_{k-1})\big)-\mb{1}_n \otimes \mb{x}^*\big\| \nonumber,\\
	\leq& \left\|\mc{A}_{\infty}\mb{x}_k-\mc{A}_{\infty}D_{\bds{\alpha}}\mc{B}_{\infty}\nabla\mb{f}\left(\mb{x}_k\right)-\left(\mb{1}_n \otimes I_p\right) \mb{x}^* \right\| \nonumber\\
	&+ \ol\beta \|\mc{A}_{\infty}\| \|\mb{x}_k-\mb{x}_{k-1}\| \nonumber\\ 
	&+ \ol{\alpha} c_{2B}\|\mc{A}_\infty\|\left\|\mb{y}_k-\mc{B}_{\infty}\mb{y}_k\right\|_{\mc{B}},
	\end{align}
where in the last inequality, we use ~$\mc{B}_{\infty}\mb{y}_k=\mc{B}_{\infty}\nabla\mb{f}\left(\mb{x}_k\right)$ adapted from Lemma~\ref{sum}. Since the last two terms in Eq.~\eqref{21} match the last two terms in Eq.~\eqref{2}, what is left is to bound the first term. Before we proceed, define
\begin{align*}
\wt{\mb{x}}_k  &\triangleq  (\bds{\pi}^\top_r \otimes I_p)\mb{x}_k,                                                                    \\
\nabla\mb{f}\left((\mb{1}_n\otimes I_p)\wt{\mb{x}}_k\right)
&\triangleq \left[\nabla f_1(\wt{\mb{x}}_k)^\top,\cdots,\nabla f_n(\wt{\mb{x}}_k)^\top\right]^\top,
\end{align*}
and note that
	\begin{align*}
	&\mc{A}_{\infty}D_{\bds{\alpha}}\mc{B}_{\infty} \\
	=& \left(\mb{1}_n \otimes I_p\right)\left(\bds{\pi}^\top_r \otimes I_p\right)
	\left(\mbox{diag}(\bds{\alpha})\otimes I_p\right)
	\left(\bds{\pi}_c \otimes I_p\right)\left(\mb{1}_n^\top \otimes I_p\right) \nonumber\\
	=& \left(\bpi_r^\top\mbox{diag}(\bds{\alpha})\bpi_c\right)\left(\mb{1}_n\otimes I_p\right)\left(\mb{1}_n^\top\otimes I_p\right).
	\end{align*}
	Now we bound the first term in Eq.~\eqref{21}. We have
    {\small\begin{align*}
	\|\mc{A}&_{\infty}\mb{x}_k-\mc{A}_{\infty}D_{\bds{\alpha}}\mc{B}_{\infty}\nabla\mb{f}(\mb{x}_k)-\left(\mb{1}_n \otimes I_p\right) \mb{x}^* \| \nonumber \\
	=&\Big\|  \left(\mb{1}_n \otimes I_p\right)\Big(\wt{\mb{x}}_k-(\bpi_r^\top\mbox{diag}(\bds{\alpha})\bpi_c)(\mb{1}_n^\top\otimes I_p) \nabla\mb{f}(\mb{x}_k)-\mb{x}^*  \Big) \Big\| \nonumber,\\
	\leq& \left\|  \left(\mb{1}_n \otimes I_p\right)\left(\wt{\mb{x}}_k -n (\bpi_r^\top\mbox{diag}(\bds{\alpha})\bpi_c) \nabla F(\wt{\mb{x}}_k)-\mb{x}^*  \right)  \right\| \nonumber\\
	&+\bpi_r^\top\mbox{diag}(\bds{\alpha})\bpi_c\left\|\left(\mb{1}_n \otimes I_p\right)\big(n\nabla F(\wt{\mb{x}}_k)- (\mb{1}_n^\top\otimes I_p) \nabla\mb{f}(\mb{x}_k)\big)  \right\| \nonumber,\\
	\triangleq&~s_1 + s_2, \nonumber
	\end{align*}}and we bound~$s_1$ and~$s_2$ next. Using Lemma~\ref{centr_d}, we have that if~$0<\bpi_r^\top\mbox{diag}(\bds{\alpha})\bpi_c<\frac{2}{nl}$,
	\begin{align}\label{s1b}
	s_1 &= \sqrt{n} \left\|\wt{\mb{x}}_k -n (\bpi_r^\top\mbox{diag}(\bds{\alpha})\bpi_c) \nabla F(\wt{\mb{x}}_k)-\mb{x}^*\right\|  \nonumber,\\
	&\leq \sqrt{n}\lambda \left\|\wt{\mb{x}}_k -\mb{x}^*\right\|  \nonumber,\\
	&= \lambda \left\|  \mc{A}_{\infty}\mb{x}_k -\mb{1}_n \otimes \mb{x}^* \right\|,
	\end{align}
	where~{\small$\lambda=\max\left\{\left|1-\mu n \bpi_r^\top\mbox{diag}(\bds{\alpha})\bpi_c\right|,
	\left|1-ln\bpi_r^\top\mbox{diag}(\bds{\alpha})\bpi_c \right|\right\}.
	$}
	We next bound~$s_2$. Since~{\small$\nabla F(\wt{\mb{x}}_k)=\frac{1}{n}(\mb{1}_n^\top\otimes I_p)\nabla\mb{f}(\wt{\mb{x}}_k)$}, 
	\begin{align}\label{23}
	s_2
	& \leq \left(\bpi_r^\top\mbox{diag}(\bds{\alpha})\bpi_c\right)n\left\|\nabla\mb{f}\left((\mb{1}_n\otimes I_p)\wt{\mb{x}}_k\right)-\nabla\mb{f}(\mb{x}_k)\right\|   \nonumber,\\ 
	&\leq
	\left(\bpi_r^\top\mbox{diag}(\bds{\alpha})\bpi_c\right)n\ol{l}c_{2\mc{A}}\left\|\mb{x}_{k}-\mc{A}_\infty\mb{x}_{k}\right\|_\mc{A}, \nonumber\\
	&\leq
	\ol{\alpha}\left(\bpi_r^\top\bpi_{c}\right)n\ol{l}c_{2\mc{A}}\left\|\mb{x}_{k}-\mc{A}_\infty\mb{x}_{k}\right\|_\mc{A},
	\end{align}
and the lemma follows from Eqs.~\eqref{s1b},~\eqref{23}, and~\eqref{21}.
\end{proof}

The next step is to bound the state difference,~$\left\|\mb{x}_{k+1}-\mb{x}_k\right\|$.
\begin{lem}\label{m}
The following inequality holds,~$\forall k$:
\begin{align}\label{3}
\|\mb{x}&_{k+1}-\mb{x}_{k}\| \nonumber\\
\leq&~\left(c_{2\mc{A}}\left\|\mc{A}-I_{np}\right\|+\ol{\alpha}c_{2\mc{A}}\ol{l}\left\|\mc{B}_{\infty}\right\|\right)\left\|\mb{x}_k-\mc{A}_\infty\mb{x}_k\right\|_\mc{A} \nonumber\\
&+\ol{\alpha}\ol{l}\left\|\mc{B}_{\infty}\right\|\|\mc{A}_\infty\mb{x}_k-\mb{1}_n \otimes \mb{x}^*\| \nonumber\\
&+\ol{\beta}\left\|\mb{x}_{k}-\mb{x}_{k-1}\right\|
+\ol{\alpha}c_{2\mc{B}}\|\mb{y}_k-\mc{B}_\infty\mb{y}_k\|_\mc{B}. \nonumber
\end{align}
\end{lem}
\begin{proof}
Note that~$\mc{A}\mc{A}_{\infty}=\mc{A}_{\infty}$ and hence~$\mc{A}\mc{A}_{\infty}-\mc{A}_{\infty}$ is a zero matrix. Following the~$\mb{x}_k$-update of~$\mc{AB}m$, we have: 
{\small\begin{align}
\|\mb{x}&_{k+1}-\mb{x}_{k}\| \nonumber\\
=& \left\| \mc{A}\mb{x}_k -D_{\bds{\alpha}}\mb{y}_k + D_{\bs\b} (\mb{x}_k-\mb{x}_{k-1}) -\mb{x}_k \right\|	\nonumber,\\
=& \left\| (\mc{A}-I_{np})(\mb{x}_k-\mc{A}_{\infty}\mb{x}_k) -D_{\bds{\alpha}}\mb{y}_k + D_{\bs\b} (\mb{x}_k-\mb{x}_{k-1}) \right\|	\nonumber,\\
\leq&~c_{2\mc{A}}\left\|\mc{A}-I_{np}\right\|\left\|\mb{x}_k-\mc{A}_{\infty}\mb{x}_k\right\|_\mc{A}+\ol{\beta}\left\|\mb{x}_{k}-\mb{x}_{k-1}\right\| + \ol{\alpha} \left\|\mb{y}_k\right\| , \nonumber
\end{align}}
and the proof follows from Lemma~\ref{y}.
\end{proof}
The final step in formulating the LTI system is to write~$\left\|\mb{y}_{k+1}-\mc{B}_{\infty}\mb{y}_{k+1}\right\|$, the biased gradient estimation error, in terms of the other three quantities. We call this biased to make a distinction with the unbiased gradient estimation error:~$\left\|\mb{y}_{k+1}-\mc{W}_\infty\mb{y}_{k+1}\right\|$, where~$\mc{W}$ is doubly-stochastic. 
\begin{lem}\label{yc}
The following inequality holds,~$\forall k$:
\begin{align}
\|\mb{y}&_{k+1}-\mc{B}_{\infty}\mb{y}_{k+1}\| \nonumber\\
=&\Big(c_{2\mc{A}}c_{\mc{B}2}~\ol{l}\left\|I_{np}-\mc{B}_{\infty}\right\|\left\|\mc{A}-I_{np}\right\| \nonumber\\
&+\ol{\alpha}c_{2\mc{A}}c_{\mc{B}2}~\ol{l}^2\left\|I_{np}-\mc{B}_{\infty}\right\|\left\|\mc{B}_{\infty}\right\|\Big)\left\|\mb{x}_k-\mc{A}_\infty\mb{x}_k\right\|_\mc{A} \nonumber\\
&+\ol{\alpha} c_{\mc{B}2}~\ol{l}^{2}\left\|I_{np}-\mc{B}_{\infty}\right\|\left\|\mc{B}_{\infty}\right\|\|\mc{A}_\infty\mb{x}_k-\mb{1}_n \otimes \mb{x}^*\| \nonumber\\
&+\ol\beta c_{\mc{B}2}~\ol{l}\left\|I_{np}-\mc{B}_{\infty}\right\|\left\|\mb{x}_{k}-\mb{x}_{k-1}\right\|\nonumber\\
&+\Big(\sigma_\mc{B}+\ol{\alpha} c_{\mc{B}2}c_{2\mc{B}}~\ol{l}\left\|I_{np}-\mc{B}_{\infty}\right\|\Big)\left\|\mb{y}_k-\mc{B}_\infty\mb{y}_k\right\|_\mc{B} \nonumber.
\nonumber
\end{align}
\end{lem}
\begin{proof}
Note that~$\mc{B}_{\infty}\mc{B}=\mc{B}_{\infty}$. From Eq.~\eqref{ABmvb}, we have:
{\begin{align} 
\|&\mb{y}_{k+1}-\mc{B}_\infty\mb{y}_{k+1}\|_\mc{B} \nonumber\\
=&~\big\|\mc{B}\mb{y}_k+\nabla \mb{f}(\mb{x}_{k+1})-\nabla \mb{f}(\mb{x}_k)
\nonumber\\
&-\mc{B}_{\infty}\big(\mb{y}_k+\nabla \mb{f}(\mb{x}_{k+1})-\nabla \mb{f}(\mb{x}_k)\big)\big\|_\mc{B} \nonumber\\
\leq&~\sigma_\mc{B}\|\mb{y}(k)-\mc{B}_\infty\mb{y}(k)\|_\mc{B} +c_{\mc{B}2}\ol{l}\left\|I_{np}-\mc{B}_{\infty}\right\|\|\mb{x}_{k+1}-\mb{x}_k\|_2, \nonumber
\end{align}}where in the inequality above we use the contraction property of~$\mc{B}$ from Lemma~\ref{contra}. The proof follows by applying the result of Lemma~\ref{m} to the inequality above.
\end{proof}

With the help of the Lemmas~\ref{xc}-\ref{yc}, we now present the main result of this paper, i.e., the~$\mc{AB}m$ algorithm converges to the global minimizer at a global~$R$-linear rate.
\begin{theorem}\label{R}
Let~$0<\bpi_r^\top\mbox{diag}(\bds{\alpha})\bpi_c<\frac{2}{nl}$, then the following LTI inequality holds entry-wise:
\begin{equation}\label{LMI}
\mb{t}_{k+1} \leq J_{\bds{\alpha},\ol\beta}\mb{t}_{k},
\end{equation}
where~$\mb{t}_{k}\in\mathbb{R}^4$ and~$J_{\bds{\alpha},\ol\b}\in\mathbb{R}^{4\times 4}$ are respectively given by:
{\small\begin{align}
\mb{t}_k&=\left[
\begin{array}{c}
\left\|\mb{x}_{k}-\mc{A}_\infty\mb{x}_{k}\right\|_\mc{A} \\
\left\|\mc{A}_\infty\mb{x}_{k}-\mb{1}_n \otimes \mb{x}^*\right\| \\
\left\|\mb{x}_{k}-\mb{x}_{k-1}\right\| \\
\left\|\mb{y}_{k}-\mc{B}_\infty\mb{y}_{k}\right\|_\mc{B}
\end{array}
\right], \nonumber\\	
J_{\bds{\alpha},\ol{\beta}}&=\left[
\begin{array}{cccc}
\sigma_\mc{A}+a_1\ol{\alpha} & a_2\ol{\alpha} &\ol{\beta} a_3 &a_4\ol{\alpha}\\
a_5\ol{\alpha} & \lambda & \ol{\beta} a_6 & a_7\ol{\alpha}\\
a_8+a_9\ol{\alpha}& a_{10}\ol{\alpha} & \ol{\beta} & a_{11} \ol{\alpha}\\
a_{12}+a_{13}\ol{\alpha}& a_{14}\ol{\alpha}& \ol{\beta} a_{15} & \sigma_\mc{B} + a_{16}\ol{\alpha}
\end{array}
\right], \nonumber
\end{align}} and the constants~$a_i$'s in the above expression are
{\small\begin{eqnarray*}
	a_1 &=& c_{\mc{A}2}c_{2\mc{A}}\ol{l}\left\|I_{np}-\mc{A}_\infty\right\| \left\|\mc{B}_{\infty}\right\|, \\
	a_2 &=& c_{\mc{A}2}\ol{l}\left\|I_{np}-\mc{A}_\infty\right\| \left\|\mc{B}_{\infty}\right\|, \\
	a_3 &=& c_{\mc{A}2}\left\|I_{np}-\mc{A}_\infty\right\|, \\
	a_4 &=& c_{\mc{A}2}c_{2\mc{B}}\left\|I_{np}-\mc{A}_\infty\right\|,\\
	a_5 &=& nc_{2\mc{A}}\left(\bpi_r^\top\bpi_c\right)\ol{l}, \\
	a_6 &=& \|\mc{A}_\infty\|, \\
	a_7 &=& c_{2\mc{B}}\|\mc{A}_\infty\|,  \\
	a_8 &=& c_{2\mc{A}}\left\|\mc{A}-I_{np}\right\|, \\
	a_{9} &=& c_{2\mc{A}}\ol{l}\left\|\mc{B}_{\infty}\right\|, \\
	a_{10} &=& \ol{l}\left\|\mc{B}_{\infty}\right\|, \\
	a_{11} &=& c_{2\mc{B}}, \\
	a_{12} &=& c_{\mc{B}2}c_{2\mc{A}}\ol{l}\left\|I_{np}-\mc{B}_{\infty}\right\|\left\|\mc{A}-I_{np}\right\|,  \\
	a_{13} &=& c_{\mc{B}2}c_{2\mc{A}}\ol{l}^2\left\|I_{np}-\mc{B}_{\infty}\right\|\left\|\mc{B}_{\infty}\right\|, \\
	a_{14} &=& c_{\mc{B}2}\ol{l}^2\left\|I_{np}-\mc{B}_{\infty}\right\|\left\|\mc{B}_{\infty}\right\|, \\
	a_{15} &=& c_{\mc{B}2}\ol{l}\left\|I_{np}-\mc{B}_{\infty}\right\|, \\
	a_{16} &=& c_{\mc{B}2}c_{2\mc{B}}\ol{l}\left\|I_{np}-B_{\infty}\right\|. 
\end{eqnarray*}}When the largest step-size,~$\ol{\alpha}$, satisfies
{\small\begin{align}
0<\ol{\alpha}& < 
\min\Bigg\{\frac{1}{nl\bpi_r^\top \bpi_{c}},\frac{\delta_3-\delta_1a_8}{a_9\delta_1+a_{10}\delta_2+a_{11}\delta_4},
\nonumber\\
&~~~\frac{(1-\sigma_B)\delta_4-\delta_{1}a_{12}}{a_{13}\delta_1+a_{14}\delta_2+a_{14}\delta_4},\frac{(1-\sigma_B)\delta_4-\delta_{1}a_{12}}{a_{13}\delta_1+a_{14}\delta_2+a_{14}\delta_4}\Bigg\} \label{a}
\end{align}}and when the largest momentum parameter,~$\ol{\beta}$, satisfies
{\small\begin{align}
0\leq\beta& <
\min\Bigg\{ 
\frac{\delta_1(1-\sigma_A)-\left(a_1\delta_1+a_2\delta_2+a_4\delta_4\right)\ol{\alpha}}{a_3\delta_3}, \nonumber\\ 
&\frac{\Big(\delta_2\mu[\bpi_r]_{\min}[\bpi_c]_{\min}-\left(a_5\delta_1+a_7\delta_4\right)\Big)\ol{\alpha}}{a_6\delta_3}, \nonumber\\
&\frac{\delta_3-\delta_1a_8-(a_9\delta_1+a_{10}\delta_2+a_{11}\delta_4)\ol{\alpha}}{\delta_3}, \nonumber\\
&\frac{(1-\sigma_B)\delta_4-\delta_{1}a_{12}-(a_{13}\delta_1+a_{14}\delta_2+a_{14}\delta_4)\ol{\alpha}}{a_{15}\delta_3}
\Bigg\}, \label{b}
\end{align}}where~$\delta_1,\delta_2,\delta_3,\delta_4$ are arbitrary constants such that
{\small \begin{align}
\left\{
\begin{array}{lll}
\delta_1 &<& \max \left\{ \frac{\delta_3}{a_8},\frac{(1-\sigma_B)\delta_4}{a_{12}} \right\}, \\
\delta_2 &>& \frac{a_5\delta_1+a_7\delta_4}{\mu[\bpi_r]_{\min}[\bpi_c]_{\min}}, \\
\delta_3 &>& 0, \\
\delta_4 &>& 0,
\end{array}
\right. \nonumber
\end{align}}then~$\rho(J_{\bds{\alpha},\ol\beta})<1$ and thus~$\|\mb{x}_k-\mb{1}_n\otimes\mb{x}^*\|$ converges to zero linearly at the rate of~$\mc{O}(\rho(J_{\bds{\alpha},\ol\beta}))^k$.
\end{theorem}
\begin{proof}
It is straightforward to verify Eq.~\eqref{LMI} by combining Lemmas~\ref{xc}-\ref{yc}. The next step is to find the range of~$\ol{\alpha}$ and~$\ol{\beta}$ such that~$\rho(J_{\bds{\alpha},\ol\beta})<1$. In the light of Lemma~\ref{rho}, we solve for a positive vector~$\bds{\delta}=[\delta_1,\delta_2,\delta_3,\delta_4]^\top$ and the range of~$\ol{\alpha}$ and~$\ol{\beta}$ such that the following inequality holds:
\begin{equation*}
J_{\bds{\alpha},\ol\beta}\bds{\delta}<\bds{\delta},
\end{equation*}
which is equivalent to the following four conditions: 
{\small \begin{align}
a_3\delta_3\beta &<\delta_1(1-\sigma_A)-\left(a_1\delta_1+a_2\delta_2+a_4\delta_4\right)\ol{\alpha}, \label{i1}\\
a_6\delta_3\beta &<\delta_2-\delta_2\lambda-\left(a_5\delta_1+a_7\delta_4\right)\ol{\alpha}, \label{i2}\\
\delta_3\beta&<\delta_3-\delta_1a_8-(a_9\delta_1+a_{10}\delta_2+a_{11}\delta_4)\ol{\alpha}, \label{i3}\\
a_{15}\delta_3\beta
 &< (1-\sigma_B)\delta_4 -\delta_{1}a_{12}-(a_{13}\delta_1+a_{14}\delta_2+a_{14}\delta_4)\ol{\alpha}. \label{i4}
\end{align}}Recall~$\lambda$ in Lemma~\ref{xo}, when~$\ol{\alpha}<\frac{1}{nl\bpi_r^\top\bpi_c}$, we have
\begin{align*}
\lambda = 1-\mu n\bpi_r^\top\mbox{diag}(\bds{\alpha})\bpi_c
\leq 1 - \mu n[\bpi_r]_{\min}[\bpi_c]_{\min}\ol{\alpha}.
\end{align*} 
Therefore, the third condition in Eq.~\eqref{i2} is satisfied when
\begin{align}
a_6\delta_3\beta < \delta_2\mu n[\bpi_r]_{\min}[\bpi_c]_{\min}\ol{\alpha}-\left(a_5\delta_1+a_7\delta_4\right)\ol{\alpha}. \label{i2'}
\end{align}
For the right hand side of the Eq.~\eqref{i1},~\eqref{i2'},~\eqref{i3} and~\eqref{i4} to be positive, each one of these equations needs to satisfy the conditions we give below.

\vspace{-0.3cm}
{\small\begin{align}\label{a1}
\mbox{Eq.~\eqref{i1}}:&~~\ol{\alpha} < \frac{\delta_1(1-\sigma_A)}{a_1\delta_1+a_2\delta_2+a_4\delta_4},\\\label{a2}
\mbox{Eq.~\eqref{i2'}}:&~~\delta_2 > \frac{a_5\delta_1+a_7\delta_4}{\mu[\bpi_r]_{\min}[\bpi_c]_{\min}},\\\label{a3}
\mbox{Eq.~\eqref{i3}}: &~~
\left\{
\begin{array}{c}
\delta_1 <~\frac{\delta_3}{a_8}, \\
\ol{\alpha} <~ \frac{\delta_3-\delta_1a_8}{a_9\delta_1+a_{10}\delta_2+a_{11}\delta_4}.
\end{array}
\right.\\\label{a4}
\mbox{Eq.~\eqref{i4}}: &~~
\left\{
\begin{array}{c}
\delta_1 <~\frac{(1-\sigma_B)\delta_4}{a_{12}}, \\
\ol{\alpha} <~  \frac{(1-\sigma_B)\delta_4-\delta_{1}a_{12}}{a_{13}\delta_1+a_{14}\delta_2+a_{14}\delta_4}.
\end{array}
\right.
\end{align}}We first choose arbitrary positive constants,~$\delta_3$ and~$\delta_4$,~then pick $\delta_1$ satisfying Eqs.~\eqref{a3} and~\eqref{a4}, and~finally choose~$\delta_2$ according to Eq.~\eqref{a2}. Note that~$\delta_1,\delta_2,\delta_3,$ and~$\delta_4$ are chosen to ensure that the upper bounds on~$\ol{\alpha}$ are all positive. Subsequently, from Eqs.~\eqref{a1},~\eqref{a3}, and~\eqref{a4}, together with the requirement that~$\ol{\alpha}<\frac{1}{nl\bpi_r^\top \bpi_{c}}$, we obtain the upper bound on the largest step-size,~$\ol{\alpha}$. Finally, the original four conditions in Eqs.~\eqref{i1},~\eqref{i2'},~\eqref{i3} and~\eqref{i4} lead to an upper bound on~$\ol{\beta}$, and the theorem follows.
\end{proof}

\vspace{-0.1cm}
\textbf{Remark 1}: In Theorem~\ref{R}, we have established the $R$-linear rate of~$\mc{AB}m$ when the largest step-size,~$\ol{\alpha}$, and the largest momentum parameter,~$\ol{\beta}$, respectively follow the upper bounds described in Eq.~\eqref{a} and Eq.~\eqref{b}. Note that~$\delta_1,\delta_2,\delta_3,\delta_4$ therein are tunable parameters and only depend on the network topology and the objective functions. The upper bounds for~$\ol\alpha$ and~$\ol\beta$ may not be computable for arbitrary directed graphs as the contraction coefficients,~$\sigma_{\mc{A}}$,~$\sigma_{\mc{B}}$, and the norm equivalence constants may be unknown. However, when the graph is undirected, we can obtain computable bounds for~$\ol\alpha$ and~$\ol\beta$, as developed in~\cite{harness,dnesterov} for example. The upper bound on~$\ol{\beta}$ also implies that if the step-sizes are relatively large, only small momentum parameters can be picked to ensure stability. 

\textbf{Remark 2}: The nonidentical step-sizes in gradient tracking methods~\cite{AugDGM,harness} have previously been studied in~\cite{AugDGM,digingstochastic,digingun,lu2018geometrical}. These works rely on some notion of heterogeneity among the step-sizes, defined respectively as the relative deviation of the step-sizes from their average,~$\frac{\|(I-W)\bs{\alpha}\|}{\|W\bs{\alpha}\|}$, in~\cite{AugDGM,digingstochastic}, and as the ratio of the largest to the smallest step-size,~${[\bds{\alpha}]_{\max}}/{[\bds{\alpha}]_{\min}}$, in~\cite{digingun,lu2018geometrical}. The authors then show that when the heterogeneity is sufficiently small and when the largest step-size follows a bound that is a function of the heterogeneity, the proposed algorithms converge to the global minimizer. It is worth noting that sufficiently small step-sizes do not guarantee sufficiently small heterogeneity in both of the above definitions. In contrast, the upper bound on the largest step-size in this paper, Eq.~\eqref{a}, is independent of any notion of heterogeneity and only depends on the objective functions and the network topology. Each agent therefore locally picks a sufficiently small step-size without any coordination. Based on the discussion in Section~\ref{s3}, our approach thus improves the analysis in~\cite{AugDGM,digingstochastic,digingun,lu2018geometrical}. Besides, Eq.~\eqref{a} allows the existence of zero step-sizes among the agents as long as the largest step-size is positive and is sufficiently small. 

\textbf{Remark 3}: To show that~$\mc{AB}m$ has an~$R$-linear rate for sufficiently small~$\ol\alpha$ and~$\ol\beta$, one can alternatively use matrix perturbation analysis as in~\cite{AB} (Theorem 1). However, it does not provide explicit upper bounds on~$\ol\alpha$ and~$\ol\beta$ in closed form. 
  
\section{Average-Consensus from~$\mc{AB}m$}\label{s6}
In this section, we show that~$\mc{AB}m$ subsumes a novel average-consensus algorithm over strongly-connected directed graphs. To show this, we choose the objective functions as$$\wt{f}_i(\mb{x})=\tfrac{1}{2}\|\mb{x}-\bds{\upsilon}_i\|^2,\quad \forall i.$$ Clearly, the minimization of~$\wt{F}=\sum_{i=1}^{n} \wt{f}_i$ is now achieved at~$\mb{x}^*=\tfrac{1}{n} \sum_{i=1}^{n}\bds{\upsilon}_i$. The~$\mc{AB}m$ algorithm, Eq.~\eqref{ABmv}, thus naturally leads to the following average-consensus algorithm, termed as~$\mc{AB}m$-$\mc{C}$, with~$\nabla\mb{f}(\mb{x}_{k+1})-\nabla\mb{f}(\mb{x}_k)=\mb{x}_{k+1}-\mb{x}_k$; for the sake of simplicity, we choose~$\alpha_i=\alpha,\beta_i=\beta,\forall i$: 
\begin{align*}
\mb{x}_{k+1} &= (\mc{A}+\b I)\mb{x}_k  - \a\mb{y}_k - \b\mb{x}_{k-1}, \\
\mb{y}_{k+1} &=  (\mc{A} + \b I - I) \mb{x}_k + (\mc{B}-\a I)\mb{y}_k - \b\mb{x}_{k-1}.
\end{align*}
Its local implementation at each agent~$i$ is given by:
\begin{align*}
\mb{x}_{k+1}^i =& \sum_{j\in\mc{N}_i\setminus i}a_{ij}\mb{x}_k^j + (a_{ii}+\b)\mb{x}_k^i - \a\mb{y}^i_k - \b\mb{x}_{k-1}^i, \\
\mb{y}_{k+1}^i =&  \sum_{j\in\mc{N}_i\setminus i}a_{ij}\mb{x}_k^j + (a_{ii}+\b-1)\mb{x}_k^i\\&+ \sum_{j\in\mc{N}_i\setminus i}b_{ij}\mb{y}_k^j + (b_{ii}-\alpha)\mb{y}_k^i - \b\mb{x}_{k-1}^i,
\end{align*}
where~$\mb{x}^i_0=\bds\upsilon_i$ and~$\mb{y}_i^0 = 0,~\forall i$. 

From the analysis of~$\mc{AB}m$, an~$R$-linear convergence of~$\mc{AB}m$-$\mc{C}$ to the average of~$\bds{\upsilon}_i$'s is clear from Theorem~\ref{R}. It may be possible to make concrete rate statements by studying the spectral radius of the following system matrix:
\begin{equation}\label{scm}
\left[\begin{array}{c}
\mb{x}_{k+1}\\ 
\mb{y}_{k+1}\\
\mb{x}_{k}
\end{array}
\right]
= \left[\begin{array}{ccc}
\mc A+\b I & -\a I & -\b I\\
\mc A+\b I - I & \mc B-\a I & -\b I\\
I & 0 & 0\\
\end{array}
\right]
\left[\begin{array}{c}
\mb{x}_{k}\\ 
\mb{y}_{k}\\
\mb{x}_{k-1}
\end{array}
\right].
\end{equation}
However, this analysis is beyond the scope of this paper. We note that when~$\b=0$, the above equations still converge to the average of~$\bds{\upsilon}_i$'s according to Theorem~\ref{R}. What is surprising is that, with~$\beta=0$,~$\mc{AB}m$-$\mc{C}$ reduces to
\begin{align}\label{sc}
\left[\begin{array}{c}
\mb{x}_{k+1}\\ 
\mb{y}_{k+1}
\end{array}
\right]
= \left[\begin{array}{ccc}
\mc A & -\a I\\
\mc A - I & \mc B-\a I
\end{array}
\right]
\left[\begin{array}{c}
\mb{x}_{k}\\ 
\mb{y}_{k}
\end{array}
\right],
\end{align}
which is surplus consensus~\cite{ac_Cai1}, after a state transformation with~$\mbox{diag}\left(I, -I\right)$; in fact, any state  transformation of the form~$\mbox{diag}(I, \wt{I})$ applies here as long as~$\wt{I}$ is diagonal (to respect the graph topology) and invertible. More importantly, compared with surplus consensus~\cite{ac_Cai1},~$\mc{AB}m$-$\mc{C}$ uses information from the past iterations. This history information is in fact the momentum from a distributed optimization perspective, which may lead to accelerated convergence as we will numerically show in Section~\ref{s7}. 

Following this discussion, choosing the local functions as~$\wt f_i$'s in~\cite{AugDGM,harness}, or in ADD-OPT~\cite{add-opt,diging}, or in FROST~\cite{linear_row,FROST}, we get average-consensus with only DS, CS, or RS weights. The protocol that results directly from~$\mc{AB}$ is surplus consensus, while the one resulting directly from FROST was presented in~\cite{ac_row}. With the analysis provided in Section~\ref{s3}, we see that the algorithm in~\cite{ac_row} is in fact related to surplus consensus after a state transformation. Clearly, \textit{accelerated} average-consensus based exclusively on either row- or column-stochastic weights can be abstracted from the discussion herein, after adding a momentum term.

\section{Numerical Experiments}\label{s7}
We now provide numerical experiments to illustrate the theoretical findings described in this paper. To this aim, we use two different graphs: an undirected graph,~$\mc{G}_1$, and a directed graph,~$\mc{G}_2$. Both graphs have~$n=500$ agents and are generated using nearest neighbor rules and then we add less than~$0.05\%$ random links. The number of edges in all cases is less~$4\%$ of the total possible edges. Since the graphs are randomly generated across experiments, two sample graphs are shown in Fig.~\ref{graph}, without the self-edges and random links for visual clarity. We generate DS weights using the Laplacian method:~$W=I-\tfrac{1}{\max_i \deg_i+1} L$, where~$L$ is the graph Laplacian and~$\deg_i$ is the degree of node~$i$. Additionally, we generate RS and CS  weights with the uniform weighting strategy:~$a_{ij}=\tfrac{1}{|\mc{N}_j^{{\scriptsize \mbox{in}}}|}$ and~$b_{ij}=\tfrac{1}{|\mc{N}_j^{{\scriptsize \mbox{out}}}|},\forall i,j$. We note that both weighting strategies are applicable to undirected graphs, while only the uniform strategy can be used over directed graphs.
\begin{figure}[!ht]
\centering
\subfigure{\includegraphics[width=1.72in]{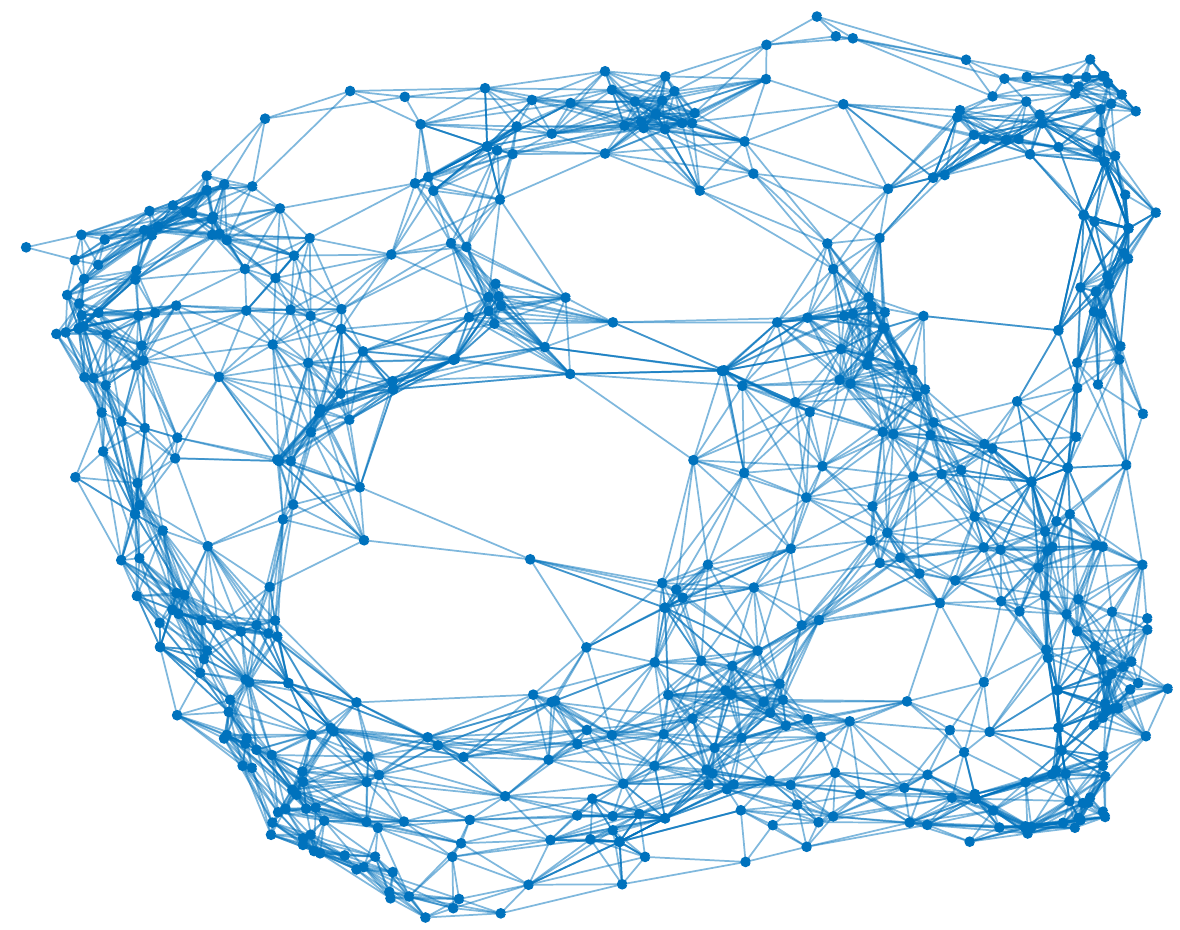}}
\subfigure{\includegraphics[width=1.72in]{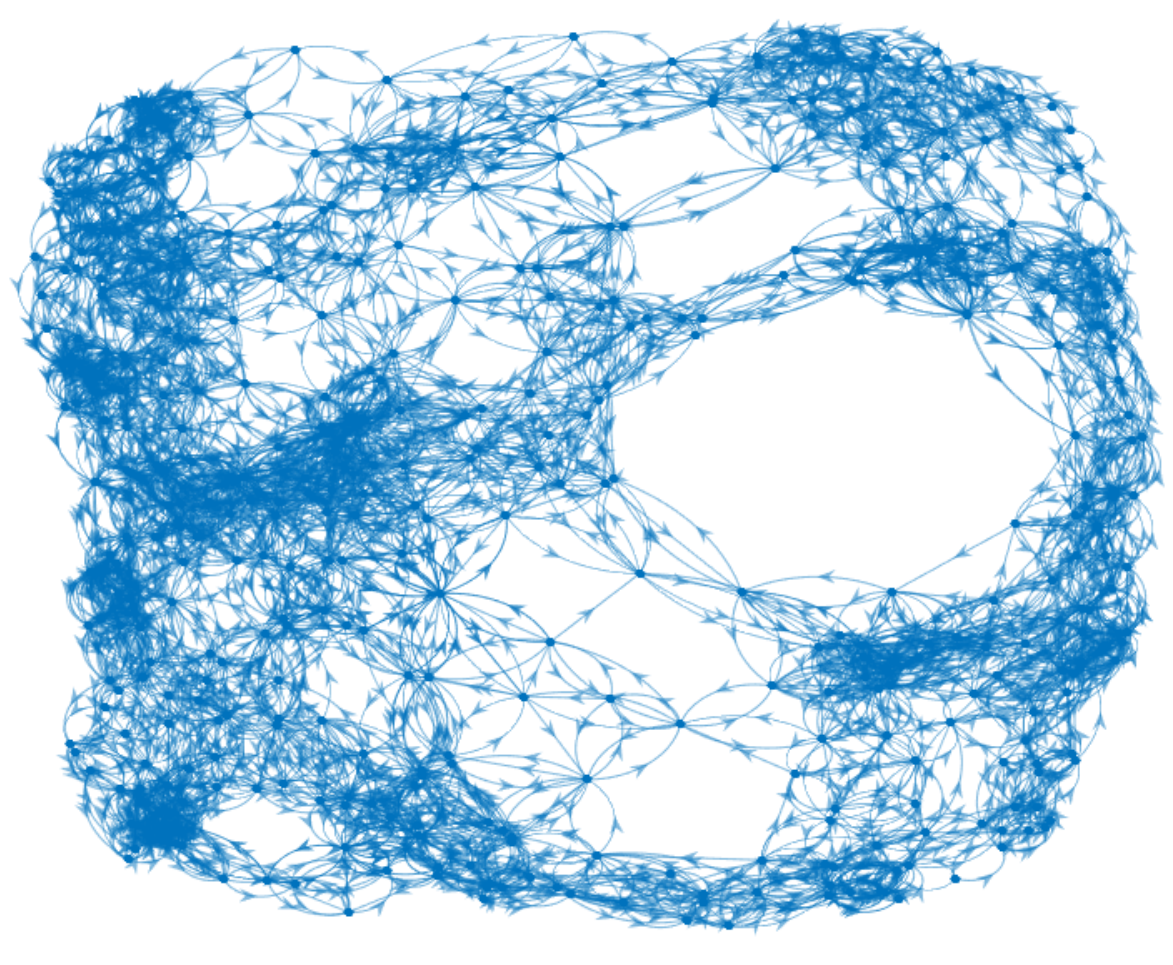}}
\caption{(Left) Undirected graph,~$\mc{G}_1$. (Right) Directed graph,~$\mc{G}_2$.}
\label{graph}
\end{figure}

\subsection{Logistic Regression}
We first consider distributed logistic regression: each agent~$i$ has access to~$m_i$ training data,~$(\mb{c}_{ij},y_{ij})\in\mathbb{R}^p\times\{-1,+1\}$, where~$\mb{c}_{ij}$ contains~$p$ features of the~$j$th training data at agent~$i$, and~$y_{ij}$ is the corresponding binary label. The agents cooperatively minimize~$F=\sum_{i=1}^nf_i(\mb{b},c)$, to find~$\mb{b}\in\mbb{R}^p,c\in\mbb{R}$, with each private loss function being

{\small\begin{equation*}
f_i(\mb{b},c)=\sum_{j=1}^{m_i}\ln\left[1+\exp\left(-\left(\mb{b}^\top\mb{c}_{ij}+c\right)y_{ij}\right)\right]+\frac{\lambda}{2}\|\mb{b}\|_2^2,
\end{equation*}}where~$\frac{\lambda}{2}\|\mb{b}\|_2^2$ is a regularization term used to prevent over-fitting of the data. The feature vectors,~$\mb{c}_{ij}$'s, are  randomly generated from a Gaussian distribution with zero mean and the binary labels are randomly generated from a Bernoulli distribution. We plot the average of residuals at each agent,~$\frac{1}{n}\sum_{i=1}^{n}\|\mb{x}_i(k)-\mb{x}^*\|_2$, and first compare the performance of the following over undirected graphs in Fig.~\ref{log_re} (Left):
\begin{inparaenum}[(i)]
\item $\mc{AB}m$ with RS and CS weights;
\item $\mc{AB}m$ with DS weights;
\item distributed optimization based on gradient tracking from~\cite{AugDGM,harness,diging}, with DS weights;
\item EXTRA from~\cite{EXTRA};
and,
\item centralized gradient descent.
\end{inparaenum} 
\begin{figure}[!h]
\centering
\subfigure{\includegraphics[width=1.72in]{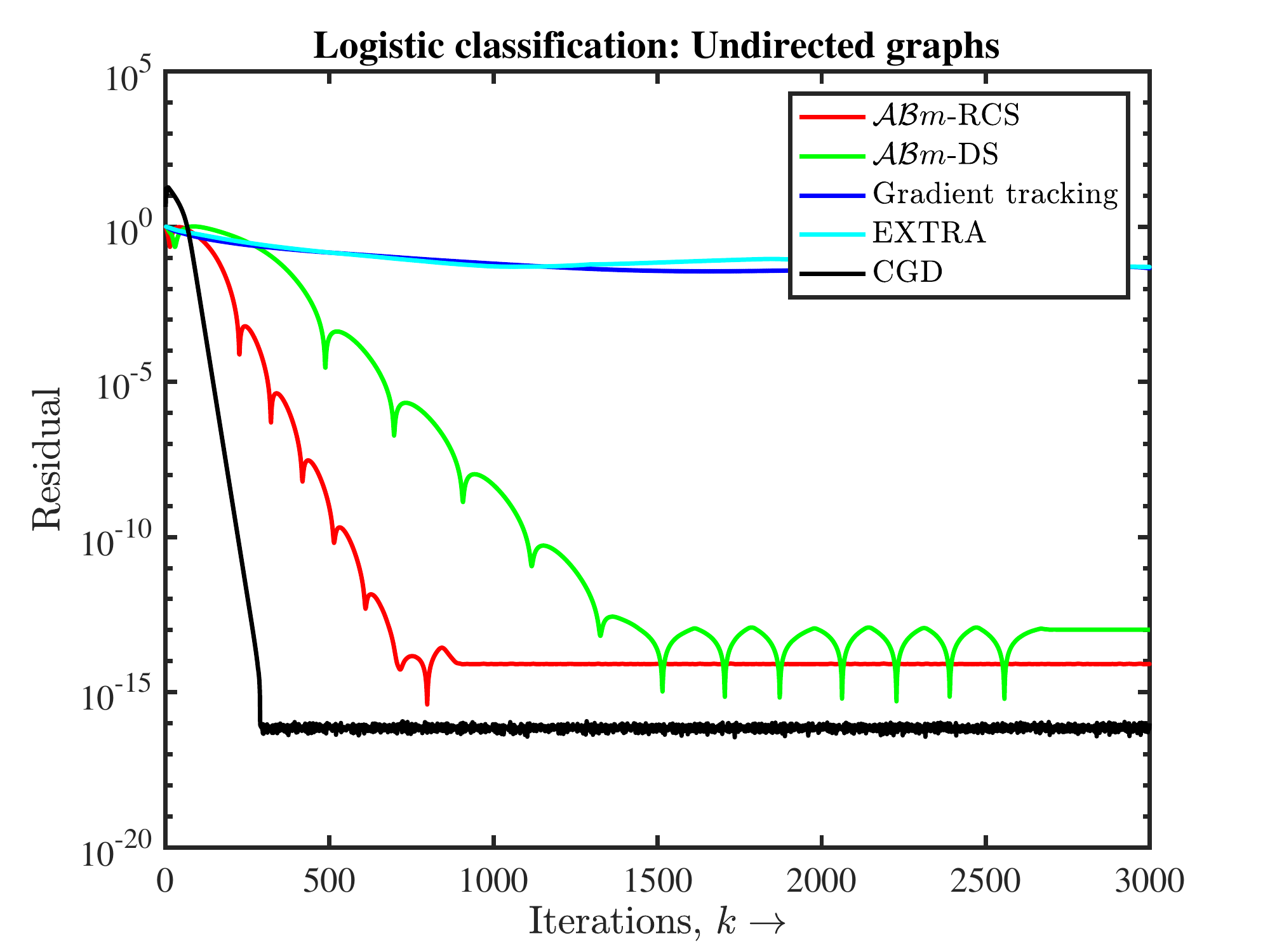}}
\subfigure{\includegraphics[width=1.72in]{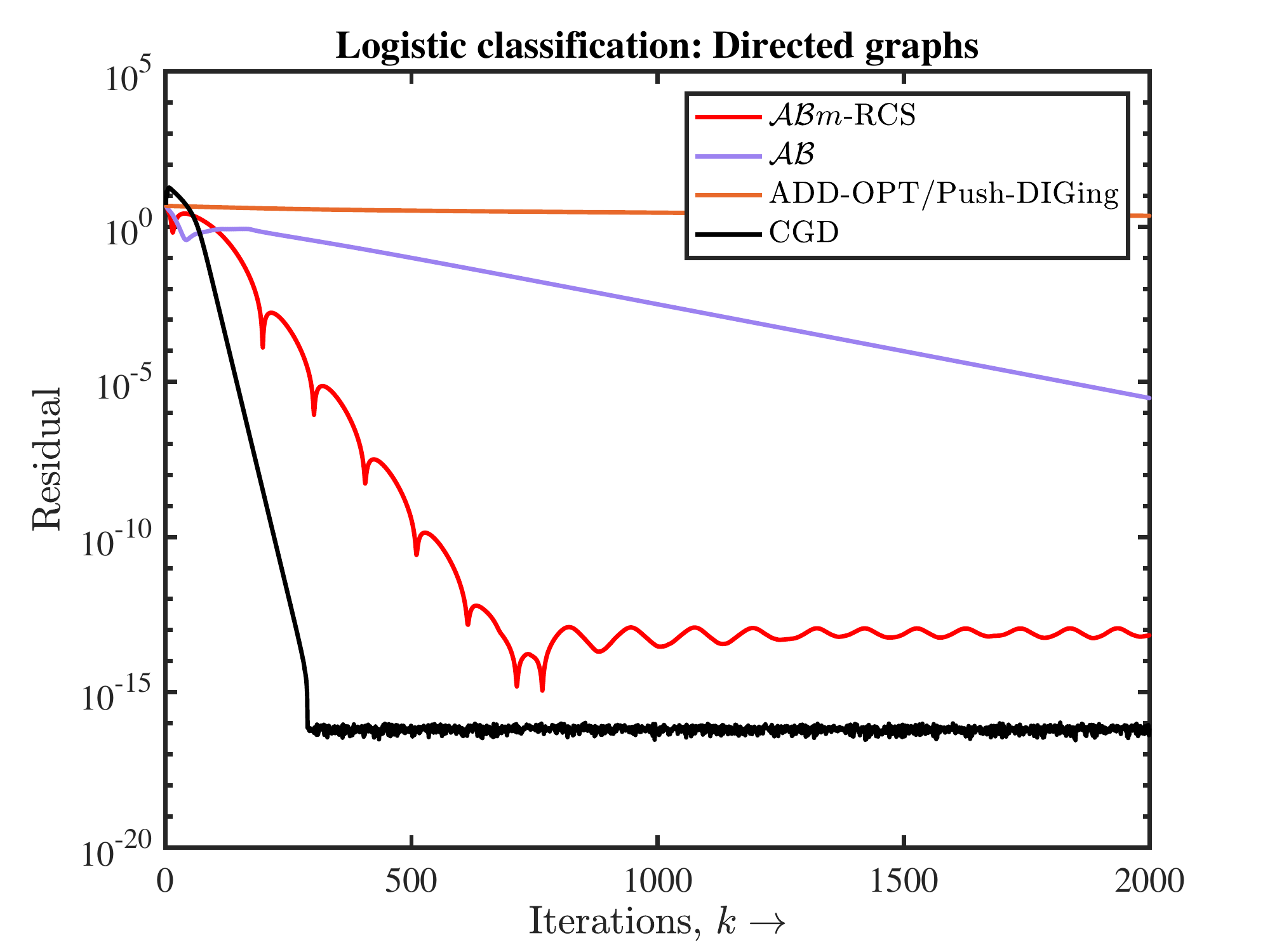}}
\caption{Logistic regression over undirected (Left) and directed graph (Right). }
\label{log_re}
\end{figure}

Next, we compare the performance similarly over \textit{directed graphs} in Fig.~\ref{log_re} (Right). Here, the algorithms with doubly-stochastic weights~\cite{EXTRA,AugDGM,harness,diging} are not applicable, and instead we compare~$\mc{AB}m$ with~$\mc{AB}$~\cite{AB}, ADD-OPT/Push-DIGing~\cite{add-opt,diging}, and centralized gradient descent. The weight matrices are chosen as we discussed before and the algorithm parameters are hand-tuned for best performance (except for gradient descent where the optimal step-size is given by~$\alpha=\tfrac{2}{\mu+l}$). We note that momentum improves the convergence when compared to applicable algorithms without momentum, while ADD-OPT/Push-DIGing are much slower because of the eigenvector estimation, see Section~\ref{s3} for details.

\begin{figure*}[!h]
\centering
\subfigure{\includegraphics[width=2.35in]{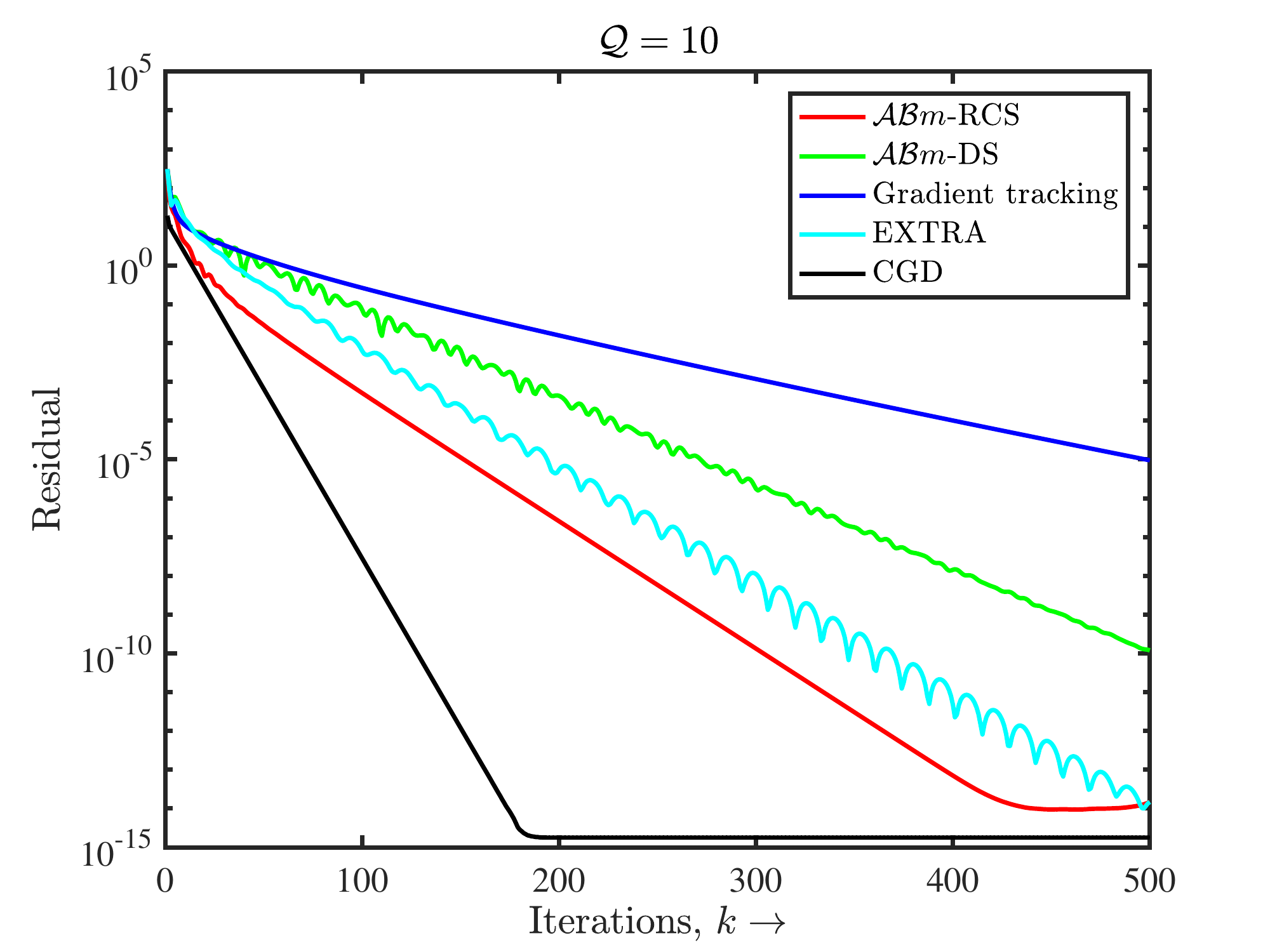}}
\subfigure{\includegraphics[width=2.35in]{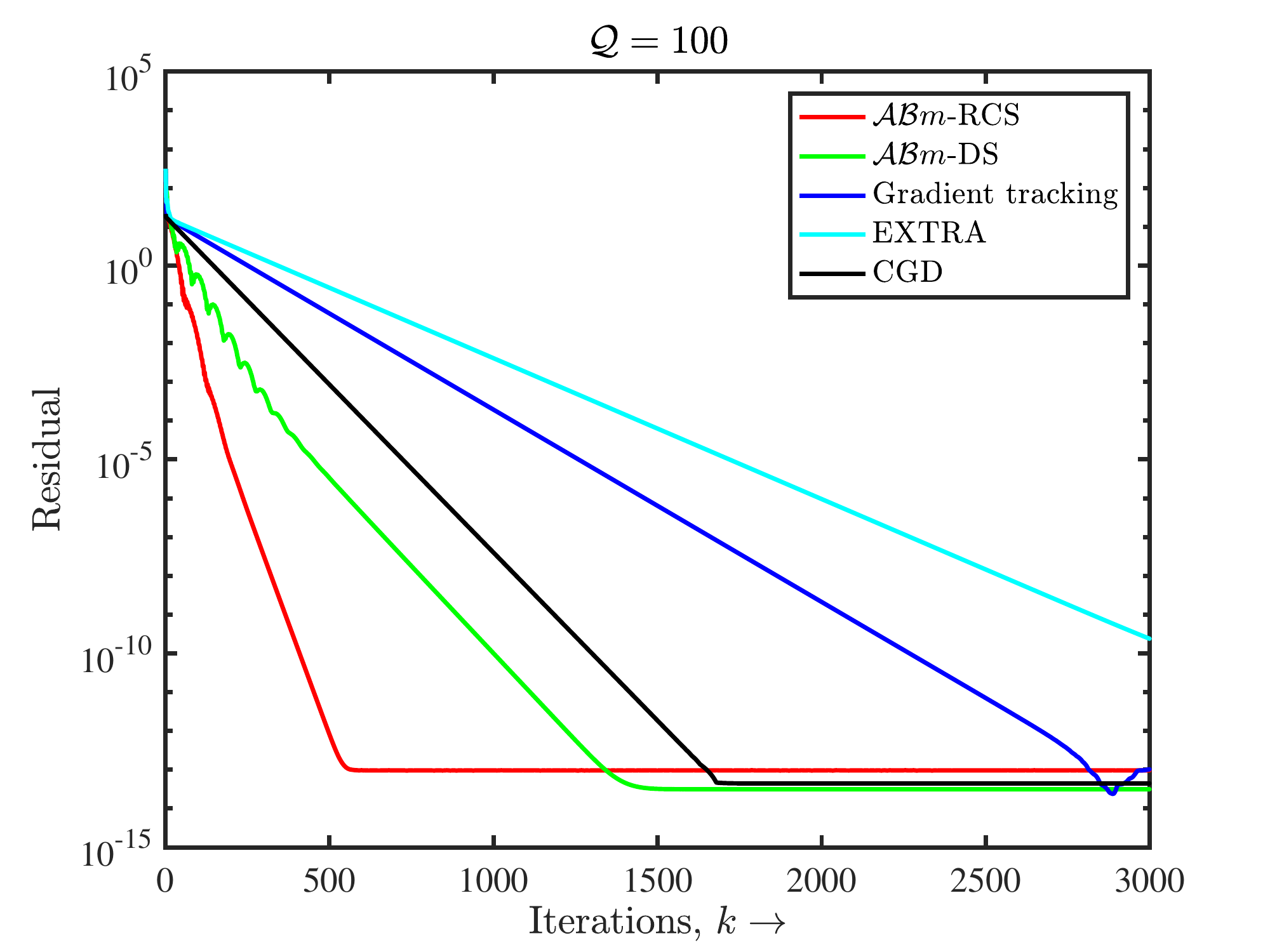}}
\subfigure{\includegraphics[width=2.35in]{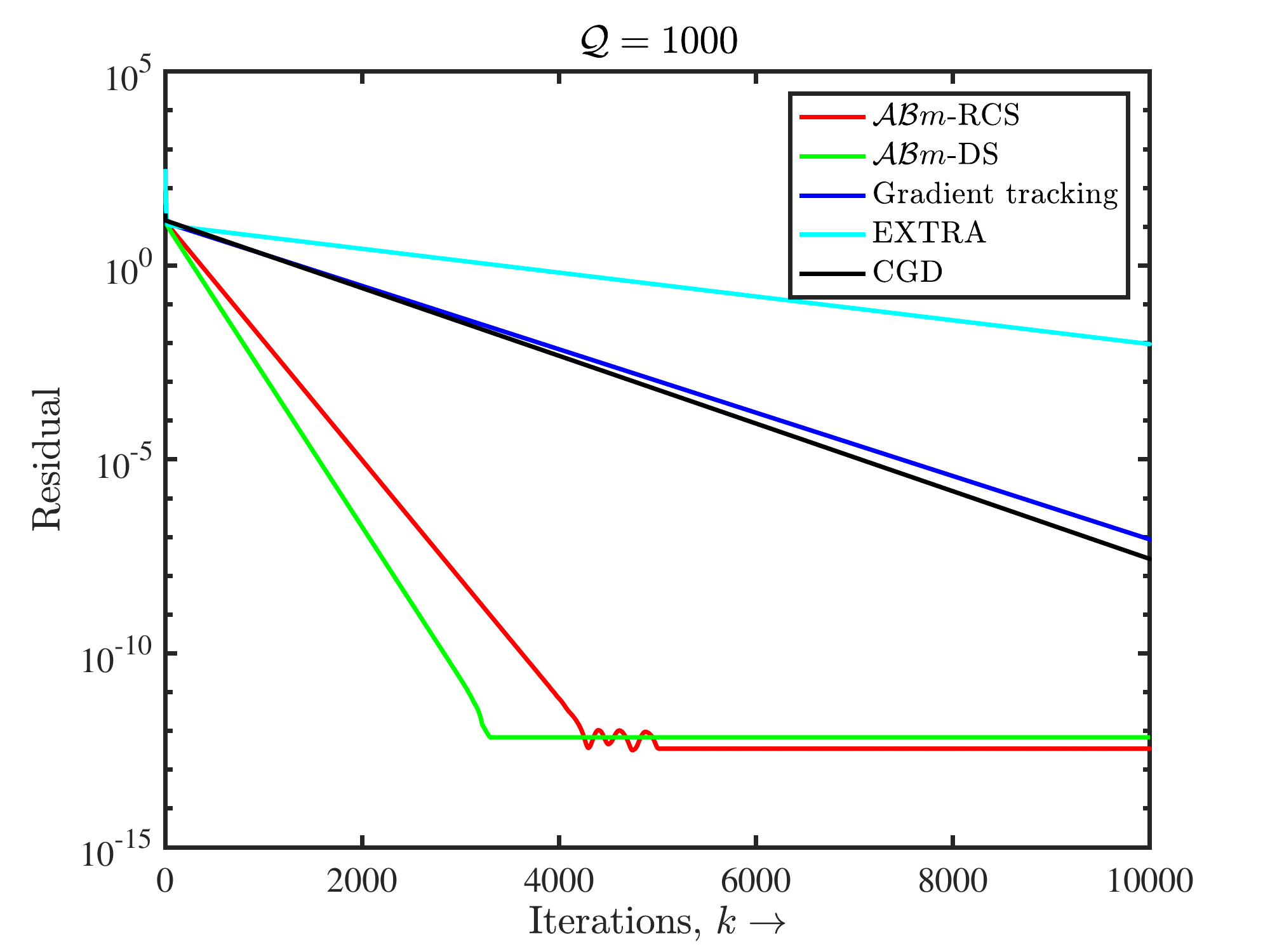}}
\caption{Performance comparison over \textit{undirected graph},~$\mc{G}_1$, as a function of the condition numbers.}
\label{Qun}
\end{figure*}
\begin{figure*}[!h]
\centering
\subfigure{\includegraphics[width=2.35in]{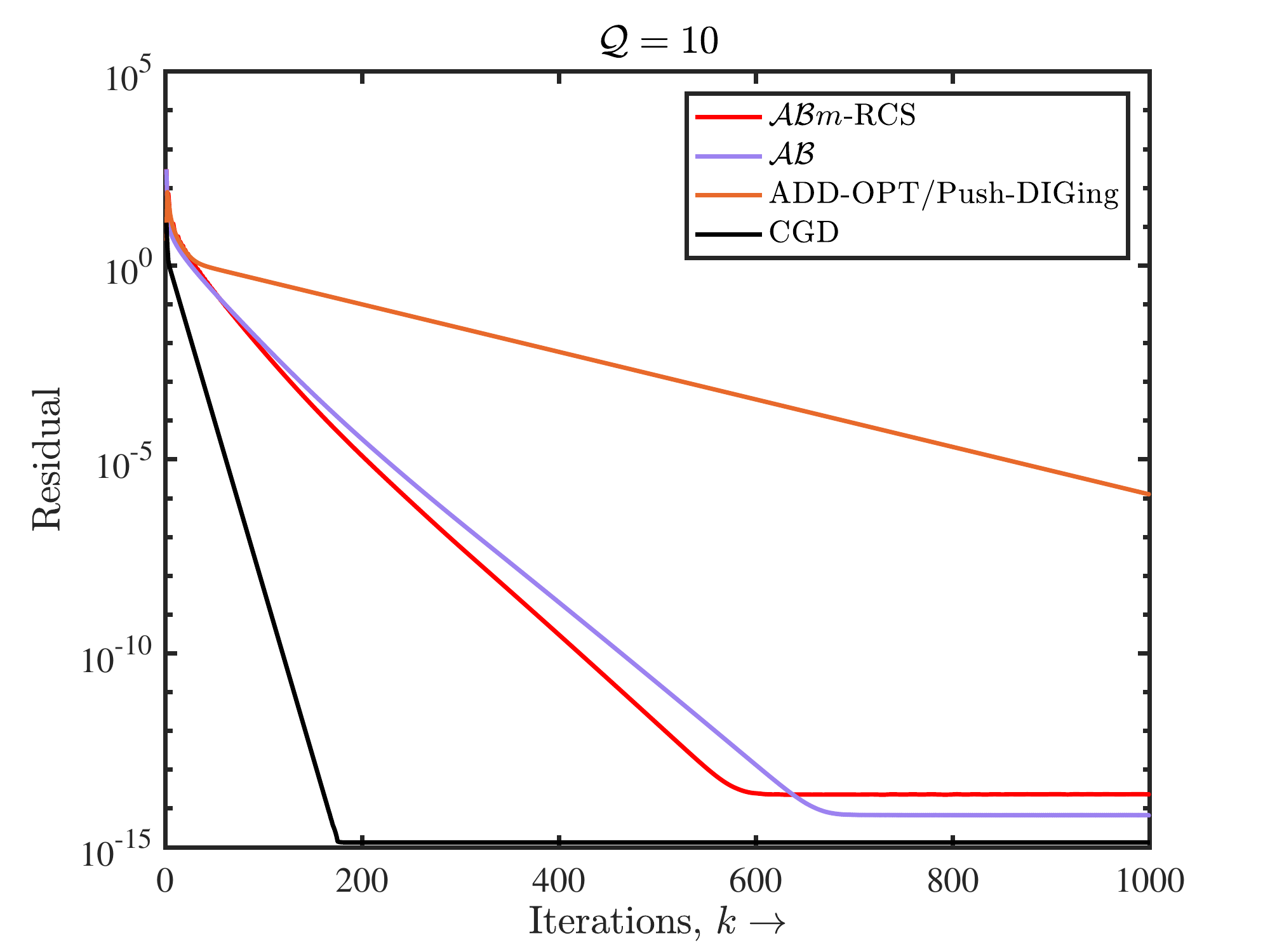}}
\subfigure{\includegraphics[width=2.35in]{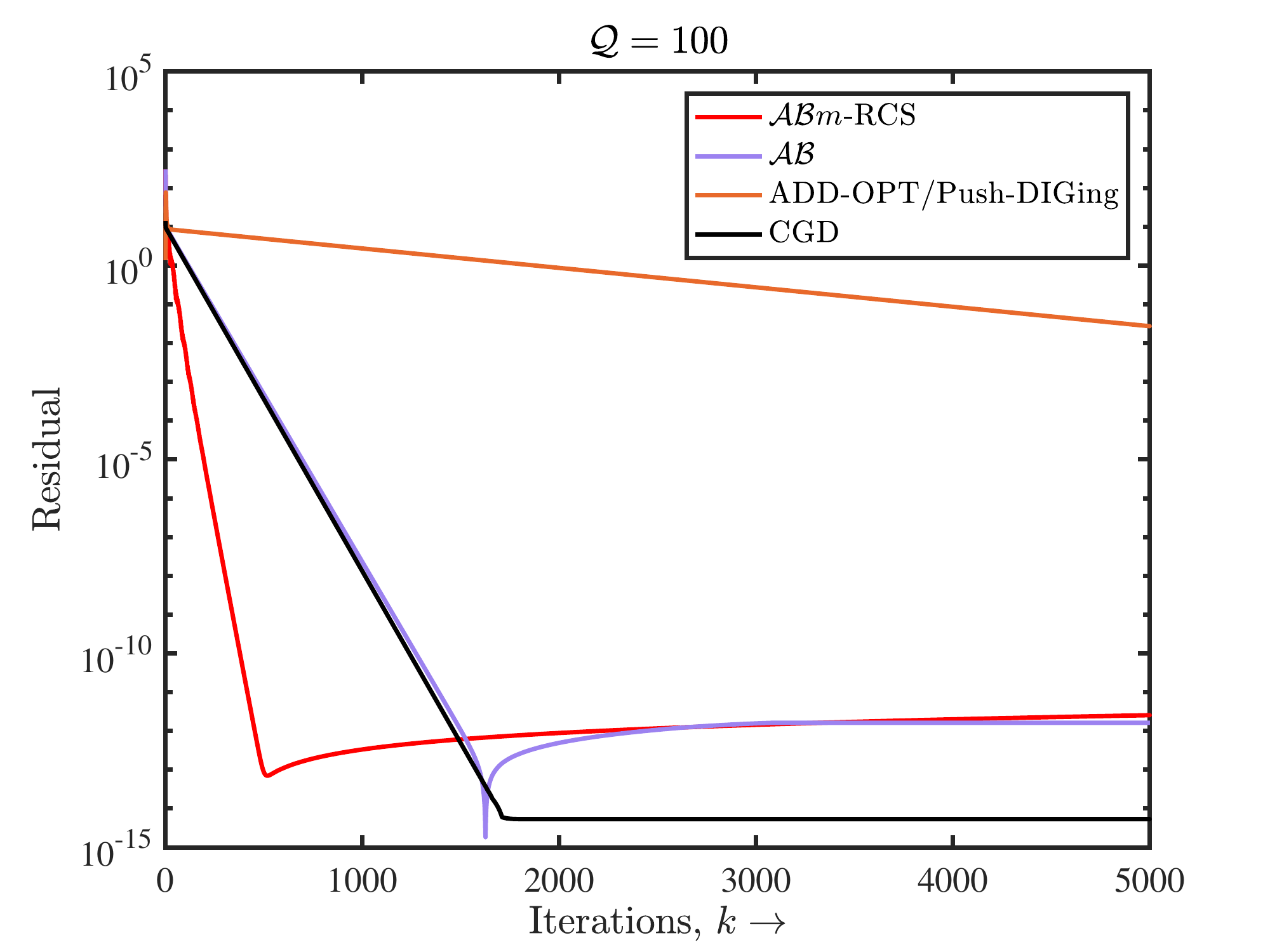}}
\subfigure{\includegraphics[width=2.35in]{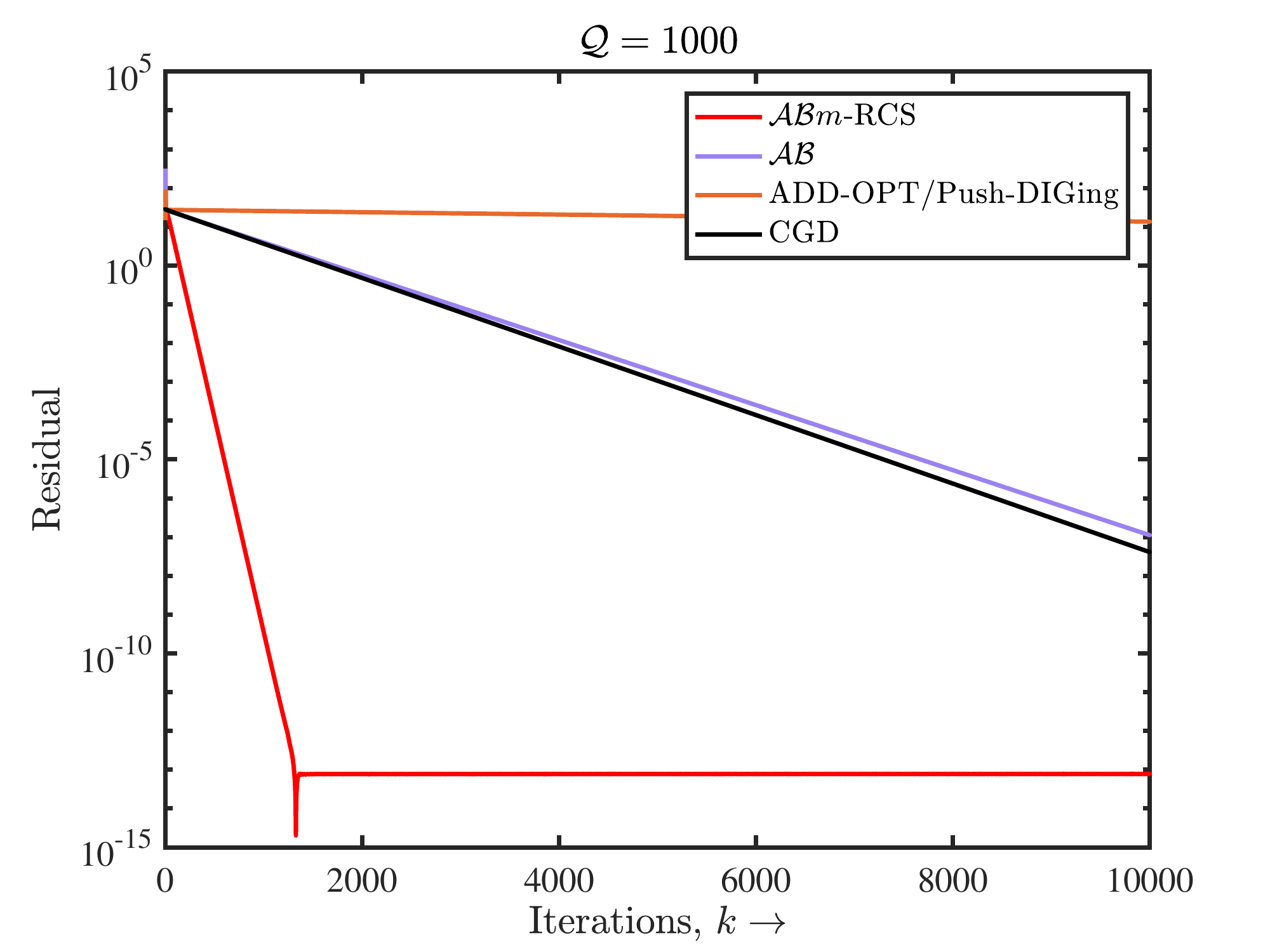}}
\caption{Performance comparison over \textit{directed graph},~$\mc{G}_2$, as a function of the condition numbers.}
\label{Qdi}
\end{figure*}

\subsection{Distributed Quadratic Programming}
We now compare the performance of the aforementioned algorithms over different condition numbers of the global objective function, chosen to be quadratic, i.e.,~$F=\sum_i \mb{x}^\top Q_i\mb{x} + \mb{b}_i^\top\mb{x}$, where~$Q_i\in\mbb{R}^{p\times p}$ is diagonal and positive-definite. The condition number~$\mc{Q}$ of~$F$ is given by the ratio of the largest to the smallest eigenvalue of~$Q\triangleq\sum_{i=1}^{n}Q_i$. We first provide the performance comparison over \textit{undirected graphs} in Fig.~\ref{Qun}, and then provide the results over \textit{directed graphs} in Fig.~\ref{Qdi}. In all of these experiments, we have hand-tuned the algorithm parameters for best performance. 

For small condition numbers, we note that  gradient descent is quite fast and the distributed algorithms suffer from a relatively slower fusion over the graphs. Recall that the optimal convergence rate of gradient decent is~$\mc{O}((\tfrac{\mc{Q}-1}{\mc{Q}+1})^k)$. When the condition number is large, gradient descent is quite conservative allowing fusion to catch up. Finally, we note that~$\mc{AB}m$, with momentum, outperforms the centralized gradient descent when the condition number is large. This observation is~consistent with the existing literature, see e.g.,~\cite{polyak1964some,polyak1987introduction,IAGM,lessard2016analysis,drori2014performance}.

\subsection{$\mc{AB}m$ and Average-Consensus}
We now provide numerical analysis and simulations to show that~$\mc{AB}m$-$\mc{C}$, in Eq.~\eqref{scm}, possibly achieves acceleration when compared with surplus-consensus, in Eq.~\eqref{sc}. To explain our choice of~$\alpha$ and~$\b$, we first note that the power limit of the system matrix in Eq.~\eqref{sc}, denoted as~$\mc{H}$, is~\cite{ac_Cai1}:
\begin{align*}
\lim_{k\rightarrow\infty}\mc{H}^k = \mc{H}_\infty =
\left[\begin{array}{ccc}
\mc{W}_\infty & -\mc{W}_\infty\\
0_{np\times np} & 0_{np\times np} 
\end{array}
\right],
\end{align*}
where~$\mc{W}_\infty=(\frac{1}{n}\mb{1}_n\mb{1}_n^\top)\otimes I_p$. It is straightforward to show that~$\mc{H}^k-\mc{H_\infty} = \left(\mc{H}-\mc{H}_\infty\right)^k.$ Similarly, for the augmented system matrix,~$\wt{\mc{H}}$, in Eq.~\eqref{scm}, we observe that

\vspace{-0.2cm}
{\small\begin{align*}
\lim_{k\rightarrow\infty}\wt{\mc{H}}^k = \wt{\mc{H}}_\infty =
\left[\begin{array}{ccc}
\mc{W}_\infty & -\mc{W}_\infty & 0_{np\times np}\\
0_{np\times np} & 0_{np\times np} & 0_{np\times np}\\
\mc{W}_\infty & -\mc{W}_\infty & 0_{np\times np}
\end{array}
\right],
\end{align*}}and it can be verified that~$\mc{\wt{H}}^k-\mc{\wt{H}_\infty} = (\mc{\wt{H}}-\mc{\wt{H}}_\infty)^k.$ We therefore use grid search~\cite{nesterov2013introductory} to choose the optimal~$\alpha^*$ in~$\mc{H}$ and the optimal~$\wt{\alpha}^*$ and~$\wt{\beta}^*$ in~$\mc{\wt{H}}$, which respectively minimize~$\rho(\mc{H}-\mc{H}_\infty)$ and~$\rho(\mc{\wt{H}}-\mc{\wt{H}}_\infty)$. Numerically, we observe  that it may be possible for the minimum of~$\rho(\mc{\wt{H}}-\mc{\wt{H}}_\infty)$ to be smaller than that of~$\rho\left(\mc{H}-\mc{H}_\infty\right)$. The convergence speed comparison between~$\mc{AB}m$-$\mc{C}$ and surplus consensus~\cite{ac_Cai1} is shown in Fig~\ref{smc} over a directed graph,~$\mc{G}_2$.
\begin{figure}[!ht]
\centering
\subfigure{\includegraphics[width=2.35in]{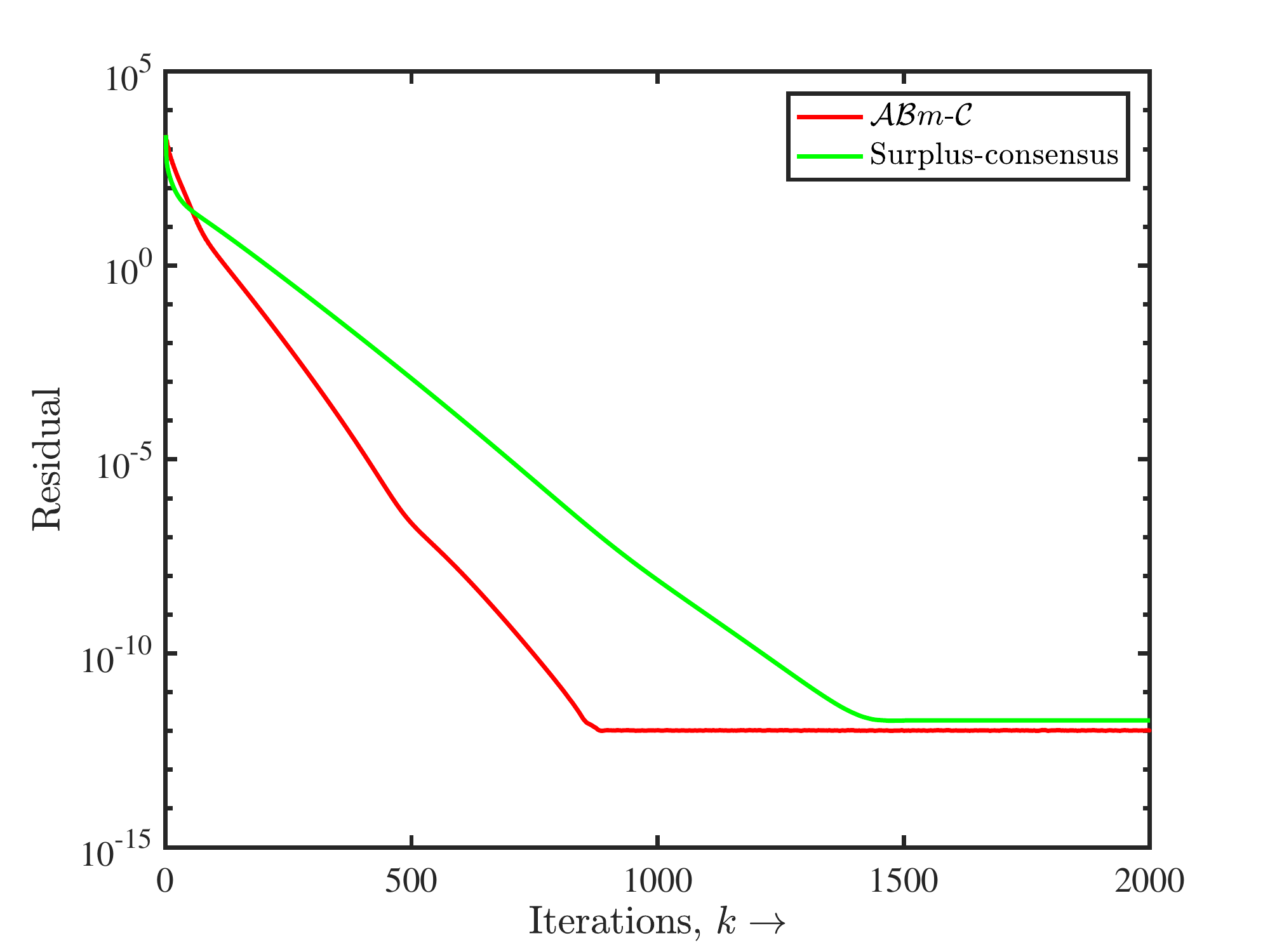}}
\caption{Average-consensus via~$\mc{AB}m$-$\mc{C}$ (with momentum) and surplus consensus (without momentum) implemented over a directed graph.}
\label{smc}
\end{figure}

 \section{Conclusions}\label{s8}
In this paper, we provide a framework for distributed optimization that removes the need for doubly-stochastic weights and thus is naturally applicable to both undirected and directed graphs. Using a state transformation based on the~non-$\mb{1}_n$ eigenvector, we show that the underlying algorithm,~$\mc{AB}$, based on a simultaneous application of both RS and CS weights, lies at the heart of several algorithms studied earlier that rely on eigenvector estimation when using only CS (or only RS) weights. We then propose the distributed heavy-ball method, termed as~$\mc{AB}m$, that combines~$\mc{AB}$ with a heavy-ball (type) momentum term. To the best of our knowledge, this paper is the first to use a momentum term based on the heavy-ball method in distributed optimization. We show that~$\mc{AB}m$ subsumes a novel average-consensus algorithm as a special case that unifies earlier attempts over directed graphs, with potential acceleration due to the momentum term.

{
	%\footnotesize
	\small
	\bibliographystyle{IEEEbib}
	\bibliography{sample}

\begin{thebibliography}{10}

\bibitem{forero2010consensus}
P.~A. Forero, A.~Cano, and G.~B. Giannakis,
\newblock ``Consensus-based distributed support vector machines,''
\newblock {\em Journal of Machine Learning Research}, vol. 11, no. May, pp.
  1663--1707, 2010.

\bibitem{distributed_Boyd}
S.~Boyd, N.~Parikh, E.~Chu, B.~Peleato, and J.~Eckstein,
\newblock ``Distributed optimization and statistical learning via the
  alternating direction method of multipliers,''
\newblock {\em Foundation and Trends in Maching Learning}, vol. 3, no. 1, pp.
  1--122, Jan. 2011.

\bibitem{raja2016cloud}
H.~Raja and W.~U. Bajwa,
\newblock ``Cloud k-svd: A collaborative dictionary learning algorithm for big,
  distributed data,''
\newblock {\em IEEE Trans. on Signal Processing}, vol. 64, no. 1, pp. 173--188,
  2016.

\bibitem{wai2018multi}
H.-T. Wai, Z.~Yang, Z.~Wang, and M.~Hong,
\newblock ``Multi-agent reinforcement learning via double averaging primal-dual
  optimization,''
\newblock {\em arXiv preprint arXiv:1806.00877}, 2018.

\bibitem{jadbabaie2003coordination}
A.~Jadbabaie, J.~Lin, and A.~Morse,
\newblock ``Coordination of groups of mobile autonomous agents using nearest
  neighbor rules,''
\newblock {\em IEEE Trans. on Automatic Control}, vol. 48, no. 6, pp.
  988--1001, 2003.

\bibitem{distributed_Mateos}
G.~Mateos, J.~A. Bazerque, and G.~B. Giannakis,
\newblock ``Distributed sparse linear regression,''
\newblock {\em IEEE Trans. on Signal Processing}, vol. 58, no. 10, pp.
  5262--5276, Oct. 2010.

\bibitem{distributed_Bazerque}
J.~A. Bazerque and G.~B. Giannakis,
\newblock ``Distributed spectrum sensing for cognitive radio networks by
  exploiting sparsity,''
\newblock {\em IEEE Trans. on Signal Processing}, vol. 58, no. 3, pp.
  1847--1862, March 2010.

\bibitem{distributed_Rabbit}
M.~Rabbat and R.~Nowak,
\newblock ``Distributed optimization in sensor networks,''
\newblock in {\em 3rd International Symposium on Information Processing in
  Sensor Networks}, Berkeley, CA, Apr. 2004, pp. 20--27.

\bibitem{safavi2018distributed}
S.~Safavi, U.~A. Khan, S.~Kar, and J.~M.~F. Moura,
\newblock ``Distributed localization: A linear theory,''
\newblock {\em Proceedings of the IEEE}, 2018.

\bibitem{DOPT1}
J.~Tsitsiklis, D.~P. Bertsekas, and M.~Athans,
\newblock ``Distributed asynchronous deterministic and stochastic gradient
  optimization algorithms,''
\newblock {\em IEEE Transactions on Automatic Control}, vol. 31, no. 9, pp.
  803--812, 1986.

\bibitem{uc_Nedic}
A.~Nedi\'{c} and A.~Ozdaglar,
\newblock ``Distributed subgradient methods for multi-agent optimization,''
\newblock {\em IEEE Trans. on Automatic Control}, vol. 54, no. 1, pp. 48--61,
  Jan. 2009.

\bibitem{duchi2012dual}
J.~C. Duchi, A.~Agarwal, and M.~J. Wainwright,
\newblock ``Dual averaging for distributed optimization: Convergence analysis
  and network scaling,''
\newblock {\em IEEE Transactions on Automatic control}, vol. 57, no. 3, pp.
  592--606, 2012.

\bibitem{ac_directed0}
D.~Kempe, A.~Dobra, and J.~Gehrke,
\newblock ``Gossip-based computation of aggregate information,''
\newblock in {\em 44th Annual IEEE Symposium on Foundations of Computer
  Science}, Oct. 2003, pp. 482--491.

\bibitem{opdirect_Tsianous}
K.~I. Tsianos, S.~Lawlor, and M.~G. Rabbat,
\newblock ``Push-sum distributed dual averaging for convex optimization,''
\newblock in {\em 51st IEEE Annual Conference on Decision and Control}, Maui,
  Hawaii, Dec. 2012, pp. 5453--5458.

\bibitem{opdirect_Nedic}
A.~Nedi\'{c} and A.~Olshevsky,
\newblock ``Distributed optimization over time-varying directed graphs,''
\newblock {\em IEEE Trans. on Automatic Control}, vol. 60, no. 3, pp. 601--615,
  Mar. 2015.

\bibitem{D-DGD}
C.~Xi, Q.~Wu, and U.~A. Khan,
\newblock ``On the distributed optimization over directed networks,''
\newblock {\em Neurocomputing}, vol. 267, pp. 508--515, Dec. 2017.

\bibitem{D-DPS}
C.~Xi and U.~A. Khan,
\newblock ``Distributed subgradient projection algorithm over directed
  graphs,''
\newblock {\em IEEE Trans. on Automatic Control}, vol. 62, no. 8, pp.
  3986--3992, Oct. 2016.

\bibitem{ac_Cai1}
K.~Cai and H.~Ishii,
\newblock ``Average consensus on general strongly connected digraphs,''
\newblock {\em Automatica}, vol. 48, no. 11, pp. 2750 -- 2761, 2012.

\bibitem{DGD_Yuan}
K.~Yuan, Q.~Ling, and W.~Yin,
\newblock ``On the convergence of decentralized gradient descent,''
\newblock {\em SIAM Journal on Optimization}, vol. 26, no. 3, pp. 1835--1854,
  Sep. 2016.

\bibitem{balancing}
A.~S. Berahas, R.~Bollapragada, N.~S. Keskar, and E.~Wei,
\newblock ``Balancing communication and computation in distributed
  optimization,''
\newblock {\em arXiv preprint arXiv:1709.02999}, 2017.

\bibitem{ADMM_Wei}
E.~Wei and A.~Ozdaglar,
\newblock ``Distributed alternating direction method of multipliers,''
\newblock in {\em 51st IEEE Annual Conference on Decision and Control}, Dec.
  2012, pp. 5445--5450.

\bibitem{ADMM_Mota}
J.~F.~C. Mota, J.~M.~F. Xavier, P.~M.~Q. Aguiar, and M.~Puschel,
\newblock ``D-{ADMM}: A communication-efficient distributed algorithm for
  separable optimization,''
\newblock {\em IEEE Trans. on Signal Processing}, vol. 61, no. 10, pp.
  2718--2723, May 2013.

\bibitem{ADMM_Shi}
W.~Shi, Q.~Ling, K~Yuan, G~Wu, and W~Yin,
\newblock ``On the linear convergence of the admm in decentralized consensus
  optimization,''
\newblock {\em IEEE Trans. on Signal Processing}, vol. 62, no. 7, pp.
  1750--1761, April 2014.

\bibitem{ESOM}
A.~Mokhtari, W.~Shi, Q.~Ling, and A.~Ribeiro,
\newblock ``A decentralized second-order method with exact linear convergence
  rate for consensus optimization,''
\newblock {\em IEEE Trans. on Signal and Information Processing over Networks},
  vol. 2, no. 4, pp. 507--522, 2016.

\bibitem{DNC}
D.~Jakoveti\'{c}, J.~Xavier, and J.~M.~F. Moura,
\newblock ``Fast distributed gradient methods,''
\newblock {\em IEEE Transactions on Automatic Control}, vol. 59, no. 5, pp.
  1131--1146, May 2014.

\bibitem{EXTRA}
W.~Shi, Q.~Ling, G.~Wu, and W~Yin,
\newblock ``Extra: An exact first-order algorithm for decentralized consensus
  optimization,''
\newblock {\em SIAM Journal on Optimization}, vol. 25, no. 2, pp. 944--966,
  2015.

\bibitem{DEXTRA}
C.~Xi and U.~A. Khan,
\newblock ``{DEXTRA: A} fast algorithm for optimization over directed graphs,''
\newblock {\em IEEE Trans. on Automatic Control}, vol. 62, no. 10, pp.
  4980--4993, Oct. 2017.

\bibitem{exactdiffusion1}
K.~Yuan, B.~Ying, X.~Zhao, and A.~H. Sayed,
\newblock ``Exact diffusion for distributed optimization and learning---part i:
  Algorithm development,''
\newblock {\em arXiv preprint arXiv:1702.05122}, 2017.

\bibitem{exactdiffusion2}
K.~Yuan, B.~Ying, X.~Zhao, and A.~H. Sayed,
\newblock ``Exact diffusion for distributed optimization and learning---part
  ii: Convergence analysis,''
\newblock {\em arXiv preprint arXiv:1702.05142}, 2017.

\bibitem{diffusion}
A.~H. Sayed,
\newblock ``Diffusion adaptation over networks,''
\newblock in {\em Academic Press Library in Signal Processing}, vol.~3, pp.
  323--453. Elsevier, 2014.

\bibitem{AugDGM}
J.~Xu, S.~Zhu, Y.~C. Soh, and L.~Xie,
\newblock ``Augmented distributed gradient methods for multi-agent optimization
  under uncoordinated constant stepsizes,''
\newblock in {\em IEEE 54th Annual Conference on Decision and Control}, 2015,
  pp. 2055--2060.

\bibitem{harness}
G.~Qu and N.~Li,
\newblock ``Harnessing smoothness to accelerate distributed optimization,''
\newblock {\em IEEE Trans. on Control of Network Systems}, Apr. 2017.

\bibitem{add-opt}
C.~Xi, R.~Xin, and U.~A. Khan,
\newblock ``{ADD-OPT}: Accelerated distributed directed optimization,''
\newblock {\em IEEE Trans. on Automatic Control}, Aug. 2017,
\newblock \textit{in press}.

\bibitem{diging}
A.~Nedi\'{c}, A.~Olshevsky, and W.~Shi,
\newblock ``Achieving geometric convergence for distributed optimization over
  time-varying graphs,''
\newblock {\em SIAM Journal of Optimization}, Dec. 2017.

\bibitem{linear_row}
C.~Xi, V.~S. Mai, R.~Xin, E.~Abed, and U.~A. Khan,
\newblock ``Linear convergence in optimization over directed graphs with
  row-stochastic matrices,''
\newblock {\em IEEE Trans. on Automatic Control}, Jan. 2018,
\newblock \textit{in press}.

\bibitem{FROST}
R.~Xin, C.~Xi, and U.~A. Khan,
\newblock ``{FROST -- Fast row-stochastic optimization with uncoordinated
  step-sizes},''
\newblock {\em Arxiv: https://arxiv.org/abs/1803.09169}, Mar. 2018.

\bibitem{AB}
R.~Xin and U.~A. Khan,
\newblock ``A linear algorithm for optimization over directed graphs with
  geometric convergence,''
\newblock {\em IEEE Control Systems Letters}, vol. 2, no. 3, pp. 325--330, Jul.
  2018.

\bibitem{dnesterov}
G.~Qu and N.~Li,
\newblock ``{Accelerated distributed Nesterov gradient descent},''
\newblock {\em Arxiv: https://arxiv.org/abs/1705.07176}, May 2017.

\bibitem{jakovetic2018unification}
D.~Jakovetic,
\newblock ``A unification and generalization of exact distributed first order
  methods,''
\newblock {\em IEEE Transactions on Signal and Information Processing over
  Networks}, 2018.

\bibitem{SUCAG}
H.-T. Wai, N.~M. Freris, A.~Nedi\'{c}, and A.~Scaglione,
\newblock ``Sucag: Stochastic unbiased curvature-aided gradient method for
  distributed optimization,''
\newblock {\em arXiv preprint arXiv:1803.08198}, 2018.

\bibitem{DAC}
M.~Zhu and S.~Mart{\'\i}nez,
\newblock ``Discrete-time dynamic average consensus,''
\newblock {\em Automatica}, vol. 46, no. 2, pp. 322--329, 2010.

\bibitem{control}
C.~A. Desoer and M.~Vidyasagar,
\newblock {\em Feedback systems: input-output properties}, vol.~55,
\newblock Siam, 1975.

\bibitem{NEXT}
P.~Di~Lorenzo and G.~Scutari,
\newblock ``Next: In-network nonconvex optimization,''
\newblock {\em IEEE Trans. on Signal and Information Processing over Networks},
  vol. 2, no. 2, pp. 120--136, 2016.

\bibitem{sonata}
Y.~Sun, G.~Scutari, and D.~Palomar,
\newblock ``Distributed nonconvex multiagent optimization over time-varying
  networks,''
\newblock in {\em Signals, Systems and Computers, 2016 50th Asilomar Conference
  on}. IEEE, 2016, pp. 788--794.

\bibitem{the_copy_work_2}
S.~Pu, W.~Shi, J.~Xu, and A.~Nedi\'{c},
\newblock ``A push-pull gradient method for distributed optimization in
  networks,''
\newblock {\em arXiv preprint arXiv:1803.07588}, 2018.

\bibitem{ac_row}
A.~Priolo, A.~Gasparri, E.~Montijano, and C.~Sagues,
\newblock ``A distributed algorithm for average consensus on strongly connected
  weighted digraphs,''
\newblock {\em Automatica}, vol. 50, no. 3, pp. 946--951, 2014.

\bibitem{digingun}
A.~Nedi\'{c}, A.~Olshevsky, W.~Shi, and C.~A. Uribe,
\newblock ``Geometrically convergent distributed optimization with
  uncoordinated step-sizes,''
\newblock in {\em IEEE American Control Conference}, May 2017.

\bibitem{digingstochastic}
J.~Xu, S.~Zhu, Y.~C. Soh, and L.~Xie,
\newblock ``Convergence of asynchronous distributed gradient methods over
  stochastic networks,''
\newblock {\em IEEE Transactions on Automatic Control}, vol. 63, no. 2, pp.
  434--448, 2018.

\bibitem{lu2018geometrical}
Q.~L{\"u}, H.~Li, and D.~Xia,
\newblock ``Geometrical convergence rate for distributed optimization with
  time-varying directed graphs and uncoordinated step-sizes,''
\newblock {\em Information Sciences}, vol. 422, pp. 516--530, 2018.

\bibitem{polyak1964some}
B.~Polyak,
\newblock ``Some methods of speeding up the convergence of iteration methods,''
\newblock {\em USSR Computational Mathematics and Mathematical Physics}, vol.
  4, no. 5, pp. 1--17, 1964.

\bibitem{polyak1987introduction}
B.~Polyak,
\newblock {\em Introduction to optimization},
\newblock Optimization Software, 1987.

\bibitem{ghadimi2015global}
E.~Ghadimi, H.~R. Feyzmahdavian, and M.~Johansson,
\newblock ``Global convergence of the heavy-ball method for convex
  optimization,''
\newblock in {\em Control Conference (ECC), 2015 European}. IEEE, 2015, pp.
  310--315.

\bibitem{IAGM}
M.~Gurbuzbalaban, A.~Ozdaglar, and P.~A. Parrilo,
\newblock ``On the convergence rate of incremental aggregated gradient
  algorithms,''
\newblock {\em SIAM Journal on Optimization}, vol. 27, no. 2, pp. 1035--1048,
  2017.

\bibitem{lessard2016analysis}
L.~Lessard, B.~Recht, and A.~Packard,
\newblock ``Analysis and design of optimization algorithms via integral
  quadratic constraints,''
\newblock {\em SIAM Journal on Optimization}, vol. 26, no. 1, pp. 57--95, 2016.

\bibitem{drori2014performance}
Y.~Drori and M.~Teboulle,
\newblock ``Performance of first-order methods for smooth convex minimization:
  a novel approach,''
\newblock {\em Mathematical Programming}, vol. 145, no. 1-2, pp. 451--482,
  2014.

\bibitem{polyak2017lyapunov}
B.~Polyak and P.~Shcherbakov,
\newblock ``Lyapunov functions: An optimization theory perspective,''
\newblock {\em IFAC-PapersOnLine}, vol. 50, no. 1, pp. 7456--7461, 2017.

\bibitem{IGM}
P.~Tseng,
\newblock ``An incremental gradient (-projection) method with momentum term and
  adaptive stepsize rule,''
\newblock {\em SIAM Journal on Optimization}, vol. 8, no. 2, pp. 506--531,
  1998.

\bibitem{sHB}
N.~Loizou and P.~Richt{\'a}rik,
\newblock ``Linearly convergent stochastic heavy ball method for minimizing
  generalization error,''
\newblock {\em arXiv preprint arXiv:1710.10737}, 2017.

\bibitem{matrix}
R.~A. Horn and C.~R. Johnson,
\newblock {\em Matrix Analysis, 2$^{\mbox{\scriptsize nd}}$ ed.},
\newblock Cambridge University Press, New York, NY, 2013.

\bibitem{nesterov2013introductory}
Y.~Nesterov,
\newblock {\em Introductory lectures on convex optimization: A basic course},
  vol.~87,
\newblock Springer Science \& Business Media, 2013.

\bibitem{bertsekas1999nonlinear}
D.~P. Bertsekas,
\newblock {\em Nonlinear programming},
\newblock Athena scientific Belmont, 1999.

\end{thebibliography}
}

\begin{IEEEbiography}[{\includegraphics[width=1in,height=1.25in,clip,keepaspectratio]{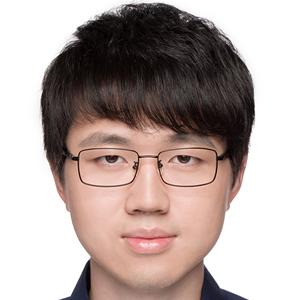}}]{Ran Xin} received his B.S. degree in Mathematics and Applied Mathematics from Xiamen University, China, in 2016, and M.S. degree in Electrical and Computer Engineering from Tufts University in 2018. Currently, he is a Ph.D. student in the Electrical and Computer Engineering department at Tufts University. His research interests include optimization theory and algorithms.
\end{IEEEbiography}

\begin{IEEEbiography}[{\includegraphics[width=1in,height=1.25in,clip,keepaspectratio]{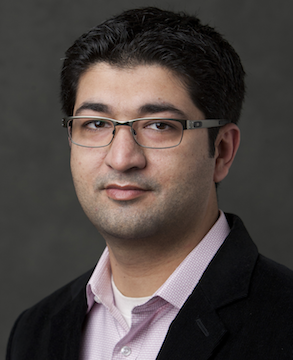}}]{Usman A. Khan} has been an Associate Professor of Electrical and Computer Engineering (ECE) at Tufts University, Medford, MA, USA, since September 2017, where he is the Director of \textit{Signal Processing and Robotic Networks} laboratory. His research interests include statistical signal processing, network science, and distributed optimization over autonomous multi-agent systems. He has published extensively in these topics with more than 75 articles in journals and conference proceedings and holds multiple patents. Recognition of his work includes the prestigious National Science Foundation (NSF) Career award, several NSF REU awards, an IEEE journal cover, three best student paper awards in IEEE conferences, and several news articles. Dr. Khan joined Tufts as an Assistant Professor in 2011 and held a Visiting Professor position at KTH, Sweden, in Spring 2015. Prior to joining Tufts, he was a postdoc in the GRASP lab at the University of Pennsylvania. He received his B.S. degree in 2002 from University of Engineering and Technology, Pakistan, M.S. degree in 2004 from University of Wisconsin-Madison, USA, and Ph.D. degree in 2009 from Carnegie Mellon University, USA, all in ECE. Dr. Khan is an IEEE senior member and has been an associate member of the Sensor Array and Multichannel Technical Committee with the IEEE Signal Processing Society since 2010. He is an elected member of the IEEE Big Data special interest group and has served on the IEEE Young Professionals Committee and on IEEE Technical Activities Board. He was an editor of the IEEE Transactions on Smart Grid from 2014 to 2017, and is currently an associate editor of the IEEE Control System Letters. He has served on the Technical Program Committees of several IEEE conferences and has organized and chaired several IEEE workshops and sessions. 
\end{IEEEbiography}

\end{document}